\documentclass[11pt]{amsart}
\usepackage{amsmath, amssymb}
\usepackage{amsfonts}
\usepackage{mathrsfs}
\usepackage[arrow,matrix,curve,cmtip,ps]{xy}

\usepackage{amsthm}
\usepackage{enumerate}

\allowdisplaybreaks

\newtheorem{theorem}{Theorem}[section]
\newtheorem{lemma}[theorem]{Lemma}
\newtheorem{proposition}[theorem]{Proposition}
\newtheorem{corollary}[theorem]{Corollary}

\newtheorem*{theorem*}{Theorem}
\theoremstyle{remark}
\newtheorem{remark}[theorem]{Remark}
\newtheorem{definition}[theorem]{Definition}

\newcommand{\myF}{E_1}
\newcommand{\myFp}{(E_1)}
\newcommand{\myG}{E_3}
\newcommand{\myGp}{(E_3)}

\newcommand{\nobox}[1]{#1}
\newcommand{\nobigzerol}{}
\newcommand{\nobigzerou}{}

\numberwithin{equation}{section}

\newcommand{\Z}{\mathbb{Z}}
\newcommand{\N}{\mathbb{N}}

\newcommand{\C}{\mathbb{C}}

\newcommand{\A}{\mathcal{A}}
\newcommand{\K}{\mathbb{K}}
\newcommand{\I}{\mathcal{I}}
\newcommand{\J}{\mathcal{J}}
\newcommand{\im}{\operatorname{im }}
\newcommand{\coker}{\operatorname{coker }}

\newcommand{\rank}{\operatorname{rank}}
\newcommand{\Ksix}{K_\textnormal{six}}
\newcommand{\Ksixo}{K^+_\textnormal{six}}

\newcommand{\fifi}{\mathbf{[11]}}
\newcommand{\fiin}{\mathbf{[1\infty]}}
\newcommand{\inin}{\mathbf{[\infty\infty]}}
\newcommand{\infi}{\mathbf{[\infty 1]}}

\newfont{\bg}{cmr9 scaled\magstep4}
\newcommand{\bigzerol}{\smash{\lower1.0ex\hbox{\bg 0}}}

\begin{document}
\title[The Ranges of $K$-theoretic Invariants]{The Ranges of $\boldsymbol{K}$-theoretic Invariants for Nonsimple Graph Algebras}

\author{S\o ren Eilers}
\address{Department for Mathematical Sciences\\University of
Copenhagen\\Universitetsparken 5\\DK-2100 Copenhagen \O \\Denmark}
\email{eilers@math.ku.dk}

\author{Takeshi Katsura}
\address{Department of Mathematics\\ Keio University\\
Yokohama, 223-8522\\ Japan}
\email{katsura@math.keio.ac.jp}

\author{Mark Tomforde}
\address{Department of Mathematics \\ University of Houston \\ Houston, TX 77204-3008 \\USA}
\email{tomforde@math.uh.edu}

\author{James West}
\address{Department of Mathematics \\ University of Houston \\ Houston, TX 77204-3008 \\USA}
\email{jdwest@math.uh.edu}

\thanks{This research was supported by the Danish National Research Foundation (DNRF) through the
Centre for Symmetry and Deformation.  Support was also provided by the NordForsk Research Network
  ``Operator Algebras and Dynamics'' (grant \#11580).  The third author was also supported by a grant from the Simons Foundation (\#210035 to Mark Tomforde).
}

\date{\today}



\begin{abstract}

There are many classes of nonsimple graph $C^*$-algebras that are classified by the six-term exact sequence in $K$-theory.  In this paper we consider the range of this invariant and determine which cyclic six-term exact sequences can be obtained by various classes of graph $C^*$-algebras.  To accomplish this, we establish a general method that allows us to form a graph with a given six-term exact sequence of $K$-groups by splicing together smaller graphs whose $C^*$-algebras realize portions of the six-term exact sequence. As rather immediate consequences, we obtain the first permanence results for extensions of graph $C^*$-algebras.

We are hopeful that the results and methods presented here will also prove useful in more general cases, such as situations where the $C^*$-algebras under investigations have more than one ideal and where there are currently no relevant classification theories available.
\end{abstract}

\maketitle

%%%%%%%%%%%%%%%%%%%%%%%%%%%%%%%%%%%%%%%%%%%%%%%%%%%%%%
\section{Introduction}
%%%%%%%%%%%%%%%%%%%%%%%%%%%%%%%%%%%%%%%%%%%%%%%%%%%%%%

In any classification program for a given class of mathematical objects there are three goals one wishes to accomplish: 
\begin{enumerate}[(1)]
\item Associate an invariant to each object in the class in such a way that the invariant completely classifies objects in the class up to some notion of equivalence.
\item Describe a tractable method to compute the invariant for a given object.
\item Determine the range of the invariant; i.e., identify all invariants that can be realized from the objects in the class. 
\end{enumerate}
In this paper we accomplish goal (3) for certain classes of graph $C^*$-algebras that are classified up to stable isomorphism by a six-term exact sequence of abelian groups in $K$-theory.

For the class of $C^*$-algebras classification programs have been very successful in the past few decades, particularly the classification of $C^*$-algebras using $K$-theoretic data as the invariant.  Two classification results were especially groundbreaking and opened several avenues for further research.  The first classification is the seminal work of Elliott in the 1970's, where it was shown that the AF-algebras are classified up to stable isomorphism by their ordered $K_0$-group \cite{Ell2}.  Later Effros, Handelman, and Shen showed that the range of this invariant is the class of all unperforated countable Riesz groups \cite{EHS}.  The second classification, occurring in the 1990's, showed that Kirchberg algebras (i.e., purely infinite, simple, separable, nuclear $C^*$-algebras in the bootstrap class) are classified up to stable isomorphism by the pair consisting of the $K_0$-group and the $K_1$-group \cite{Kir3, Phi}.  Moreover, it has been shown (cf.\ \cite[4.3.3]{mr:cnsc}) that the range of this invariant is all pairs $(G_0, G_1)$ of countable abelian groups.

In recent years more attention has been paid to the classification of nonsimple $C^*$-algebras \cite{Ro6, MN, Res}, and in this paper we focus on the classification of graph $C^*$-algebras.  It is known that any simple graph $C^*$-algebra is either an AF-algebra or a Kirchberg algebra, and hence is classified by either Elliott's Theorem or the Kirchberg-Phillips Classification.  
In particular, for a simple graph $C^*$-algebra the pair $(K_0(C^*(E)), K_1(C^*(E)))$ is a complete stable isomorphism invariant, where we view $K_0(C^*(E))$ as a pre-ordered group.  

A calculation for the $K$-theory \nobox{without order} was determined by Raeburn and Szyma\'nski \cite[Theorem~3.2]{RS} and by Drinen and the third author \cite[Theorem~3.1]{DT2}, and it is a consequence of these results that the $K_1$-group of a graph $C^*$-algebra must be a free abelian group. 
The order of the $K_0$-group was completely determined in \cite{amp} and \cite{Tomforde}.
The range of this invariant for simple graph $C^*$-algebras was calculated in independent work of Drinen and Szyma\'nski.  Drinen showed that any AF-algebra is stably isomorphic to a graph $C^*$-algebra of a row-finite graph, and it follows from this that all simple Riesz groups are attained as the $K_0$-group of a simple AF graph $C^*$-algebra.  Szyma\'nski \cite{Szy2} proved that for simple graph $C^*$-algebras that are Kirchberg algebras, all pairs of countable abelian groups $(G_0, G_1)$ with $G_1$ free abelian may be attained, and moreover for any such pair one may choose a graph $C^*$-algebra associated with a graph that is row-finite, transitive, and has a countably infinite number of vertices.

In classification efforts for the nonsimple graph $C^*$-algebras, the first and third author have shown that if $C^*(E)$ is a graph $C^*$-algebra with a unique proper nontrivial ideal $\I$, then the six-term exact sequence in $K$-theory 
$$ \xymatrix{
K_0(\I)\ar[r]^-{\iota_*} & K_0(C^*(E)) \ar[r]^-{\pi_*} & K_0 (C^*(E)/\I) \ar[d]^-{\partial_0}\\
K_1 (C^*(E)/\I) \ar[u]^-{\partial_1} & K_1 (C^*(E))  \ar[l]^-{\pi_*} & K_1(\I) \ar[l]^-{\iota_*}} $$
is a complete stable isomorphism invariant \cite[Theorem~4.5]{ET1}.  Additionally, it was shown in \cite[Theorem~4.7]{ET1} that this six-term exact sequence is also a complete stable isomorphism invariant in the case when $\I$ is a largest ideal in the graph $C^*$-algebra and $\I$ is an AF-algebra.  More recently \cite{ERR}, the first author, Restorff, and Ruiz have shown that the six-term exact sequence is also a complete stable isomorphism invariant in the case when $\I$ is a smallest ideal in the graph $C^*$-algebra $C^*(E)$, and $C^*(E)/\I$ is an AF-algebra. Thus there are several classes of graph $C^*$-algebras for which the six-term exact sequence arises as a complete stable isomorphism invariant.  Consequently, these results accomplish goal (1) for several classes of nonsimple graph $C^*$-algebras, and even more results have been announced recently.
With regards to goal (2), the computation of the invariant, it was shown by the first author, Carlsen, and the third author that the six-term exact sequence can be calculated from data provided by the vertex matrix of the graph \cite[Theorem~4.1]{CET}.

In this paper we turn our attention to goal (3) of classification: computing the range of the six-term exact sequence.  Following \cite{goodhand} (cf.\ \cite{ERR2}) we focus our attention to \textsl{stenotic} extensions given by an ideal $\I$ that contains, or is contained in, any other ideal of the $C^*$-algebra $C^*(E)$ in question.  We note that the classification results mentioned above cover all such extensions where ideal and quotient are either AF or simple.  Basic results regarding the $K$-theory of extensions given by graph $C^*$-algebras lead to the conditions that the $K_1$-groups must be free abelian groups, that the connecting map from $K_0(C^*(E)/\I)$ to $K_1(\I)$ must vanish, and when the extensions are stenotic  with ideals and quotients that are either AF or simple we obtain
certain further restrictions on the order of $K_0(C^*(E))$. We prove that these are the only restrictions, and combine our results with classification results to obtain the first permanence results for graph $C^*$-algebras, which allow us to determine from inspection of the $K$-theory when a given stable extension of AF or simple graph$C^*$-algebras is itself a graph $C^*$-algebra. As far as we can tell, this result is the first of its kind, even when restricted to the classical case of Cuntz-Krieger algebras where the necessary classification theory has been available since \cite{withgunnar}.
We also provide complete results on the case when $C^*(E)$ is unital, determining the possible position of the order unit of $K_0(C^*(E))$ in this case.

This paper is organized as follows. In Section~\ref{prelim-sec}
we establish some preliminaries and explain a key observation for our main results, namely, that for a graph $C^*$-algebra the six-term exact sequence from $K$-theory may be obtained by applying the Snake Lemma to a certain commutative diagram determined by the vertex matrix of the graph. In Section~\ref{simple-realize-sec} we obtain some range results for the $K$-theory of simple graph $C^*$-algebras and AF graph $C^*$-algebras.  In particular, we prove that there exist graph $C^*$-algebras from various classes (e.g., stable Kirchberg algebras, AF-algebras, simple Cuntz-Krieger algebras, unital Kirchberg algebras) that realize various $K_0$-groups and $K_1$-groups.  We also need a slightly stronger version of Szyma\'nski's Theorem \cite[Theorem~1.2]{Szy2}.  We are able to give a new, shorter proof of Szyma\'nski's Theorem (see Theorem~\ref{realize-simple-K-theories-prop}) that allows us to choose a graph with the additional properties that we need.  In Section~\ref{comm-diagram-sec} we establish the main technical results of our paper.  We prove a homological algebra result in Proposition~\ref{piece-together-prop} that allows us to produce various six-term exact sequences using the Snake Lemma and develop methods to arrange for positivity of the block matrices found there. Recasting our findings in the context of graphs in Section~\ref{gluing}, we provide sufficient conditions on a pair of graphs $E_1,E_3$ with $K$-theory fitting in a given six-term exact sequence to ensure 
in Proposition~\ref{six-term-graph-splice} that a new graph $E_2$ may be created in a way realizing the given $K$-theory. This new graph is formed by taking the disjoint union of $E_1$ and $E_3$, and then drawing a number of edges from vertices in $E_3$ to vertices in $E_1$ as determined by Proposition~\ref{piece-together-prop}. We also analyze the pre-order on the $K_0$-group of a graph $C^*$-algebra in terms of the pre-order on the $K_0$-groups of an ideal and its quotient. 
In Section~\ref{ranges-sec} we present the main results of this paper.  We apply all of our results to calculate the range of the six-term exact sequence in $K$-theory for various classes of graph $C^*$-algebras.  We show that for graph $C^*$-algebras with a unique proper nontrivial ideal we are able to attain any six-term exact sequence satisfying the obvious obstructions mentioned earlier.  We also determine exactly which six-term exact sequences are obtained in various other classes, including Cuntz-Krieger algebras with a unique proper nontrivial ideal, graph $C^*$-algebra extensions of unital Kirchberg algebras, graph $C^*$-algebras with a largest ideal that is AF, and graph $C^*$-algebras with a smallest ideal whose quotient is AF.  In all these cases, the proof is obtained by using results from Section~\ref{simple-realize-sec} to obtain graphs whose $C^*$-algebras realize portions of the six-term exact sequence, and then using Proposition~\ref{six-term-graph-splice} to splice these graphs together into a larger graph whose $C^*$-algebra has the required invariant. Finally, in Section~\ref{permanence} we show how to obtain permanence results from our results, giving a complete description in several cases of when an extension of two graph $C^*$-algebras is again a graph $C^*$-algebra.

The authors wish to thank Mike Boyle for inspiring discussions in the early phases of this work, and the authors would also like to thank Efren Ruiz for suggesting improvements in the later phases.

%%%%%%%%%%%%%%%%%%%%%%%%%%%%%%%%%%%%%%%%%%%%%%%%%%%%%%
\section{Preliminaries} \label{prelim-sec}
%%%%%%%%%%%%%%%%%%%%%%%%%%%%%%%%%%%%%%%%%%%%%%%%%%%%%%

For a countable set $X$, 
we let $\Z^X$ denote the free abelian group 
generated by the basis $\{\delta_x\}_{x\in X}$ indexed by $X$. 
For $n \in \N$, we denote $\Z^{\{1,2,\ldots, n\}}$ by $\Z^n$ 
which is the direct sum of $n$ copies of $\Z$, 
and denote $\Z^{\N}$ by $\Z^\infty$, 
which is the direct sum of countably infinite copies of $\Z$.

For two countable sets $X$ and $Y$, 
a \emph{column-finite} $Y \times X$ matrix with entries in $\Z$
is $A= (a_{y,x})_{(y,x) \in Y \times X}$ with $a_{y,x} \in \Z$ 
such that for each $x \in X$ there are only finitely many $y$ 
with $a_{y,x}\neq 0$. 
The collection of all such matrices is denoted by $M_{Y,X}(\Z)$. 
If $X=Y$, we denote $M_{X,X}(\Z)$ by $M_{X}(\Z)$. 
As above, we use notations like 
$M_{m,n}(\Z)$ for $m,n \in \{1, 2, \ldots, \infty \}$ 
which is the collection of all column-finite $m \times n$ matrices 
with entries in $\Z$. 

For two countable sets $X$ and $Y$, 
there is a one-to-one correspondence 
between elements of $M_{Y,X}(\Z)$ and $\Z$-module maps from $\Z^X$ to $\Z^Y$:  
Each matrix $A= (a_{y,x})$ in $M_{Y,X} (\Z)$ 
corresponds to a $\Z$-module map $\phi$ from $\Z^X$ to $\Z^Y$ 
by $\phi (\delta_x) = \sum_{y \in Y} a_{y,x} \delta_y$ 
which makes sense by the column-finite condition.
When we have a matrix $A$ we will often identify the matrix itself 
with the corresponding $\Z$-module map, using the notation $A$ for both, 
and for $\xi \in \Z^X$ we will often write $A\xi$ in place of $A(\xi)$.
If $X$ is a subset of $Y$, 
we denote by $I \in M_{Y,X}(\Z)$ the map defined 
by $I(\delta_x)=\delta_x$ for every $x \in X$, 
and by $P \in M_{X,Y}(\Z)$ the map defined 
by $P(\delta_x)=\delta_x$ for every $x \in X$ 
and $P(\delta_y)=0$ for every $y \in Y \setminus X$. 

We note that the trivial abelian group $\{0\}$ is denoted by $0$, 
and the unique $\Z$-module map from or to $0$ is also denoted by $0$. 
We have $\Z^{\emptyset}=\Z^0=0$ 
and for a countable set $X$ 
the $X \times \emptyset$ matrix and 
the $\emptyset \times X$ matrix corresponding to 
the unique $\Z$-module maps $0$ are denoted by $\emptyset$. 
Thus $M_{\emptyset,X} (\Z)=\{\emptyset\}$ 
and $M_{X,\emptyset} (\Z)=\{\emptyset\}$ 
by definition.

\subsection{Extension and $K$-theory preliminaries}

In this paper, an \emph{ideal} of a $C^*$-algebra will
we mean a closed two-sided ideal. 
Every nonzero $C^*$-algebra $\A$ has at least two ideals; 
$0$ and $\A$ which are called \emph{trivial ideals}. 
Hence a \emph{nontrivial ideal} is an ideal that is nonzero 
and proper,
and a simple $C^*$-algebra has not nontrivial ideals). 

\begin{definition}
An ideal $\I$ of $\A$ is said to be a \emph{smallest ideal} 
(resp.\ a \emph{largest ideal}) 
if $\I$ is nontrivial and 
whenever $\J$ is a nontrivial ideal of $\A$, 
we have $\I \subseteq \J$ (resp.\ $\J \subseteq \I$). An ideal $\I$ of $\A$  is called \emph{stenotic} if for any ideal $\J$, either $\J\subseteq \I$ or $\I\subseteq \J$. 
\end{definition}

The notion of stenosis was introduced to $C^*$-algebras in \cite{goodhand}. Obviously any smallest or largest ideal is stenotic.
Note that 
if $\I$ is a smallest (resp.\ largest) ideal 
then $\I$ (resp.\ $\A/\I$) is simple, 
but the converses are not true in general. 
A $C^*$-algebra $\A$ need not have a smallest ideal 
nor a largest ideal, but if it does, it will be unique. 
The uniqueness shows 
the following easy but useful observation. 

\begin{lemma} \label{trivial-lemma}
Let $\A$ and $\A'$ be $C^*$-algebras. 
Suppose both $\A$ and $\A'$ have smallest (or largest) ideals $\I$ and $\I'$. 
Then $\A$ and $\A'$ are isomorphic if and only if 
there exists a commutative diagram 
\begin{equation*} 
\xymatrix{
0 \ar[r] & \I \ar[r] \ar[d]& \A \ar[r] \ar[d]& \A/\I \ar[r] \ar[d]& 0\\
0 \ar[r] & \I' \ar[r] & \A' \ar[r] & \A'/\I' \ar[r] & 0
}
\end{equation*}
in which the three vertical maps are isomorphisms. 
\end{lemma}

If a $C^*$-algebra $\A$ has a unique nontrivial ideal $\I$, 
then $\I$ is smallest and largest 
(conversely a smallest and largest ideal is a unique nontrivial ideal). 
Thus we have an analogous result 
for $C^*$-algebras having unique nontrivial ideals. 

One advantage of Lemma~\ref{trivial-lemma} 
is that a short exact sequence gives us a powerful invariant. 
From a short exact sequence
\[
\xymatrix{
0 \ar[r] & \I \ar[r]^{\iota} & \A \ar[r]^{\pi} & \A/\I \ar[r] & 0
}
\]
of $C^*$-algebras, 
$K$-theory gives a cyclic six-term exact sequence 
\begin{equation} \label{Ksix}
\xymatrix{
K_0(\I)\ar[r]^-{\iota_*} & K_0(\A) \ar[r]^-{\pi_*} & K_0 (\A/\I) \ar[d]^-{\partial_0}\\
K_1 (\A/\I) \ar[u]^-{\partial_1} & K_1 (\A)  \ar[l]^-{\pi_*} & K_1(\I) \ar[l]^-{\iota_*}
} 
\end{equation}
of abelian groups.  
For convenience of notation, 
we let $\Ksix(\A,\I)$ denote this cyclic six-term exact sequence 
of abelian groups.

If we have a cyclic six-term exact sequence $\mathcal{E}$ of the form
\begin{equation} \label{six-term}
\xymatrix{
G_1 \ar[r]^-{\epsilon} & G_2 \ar[r]^-{\gamma} & G_3 \ar[d]^-{\delta_0} \\
F_3 \ar[u]^-{\delta_1} & F_2 \ar[l]_-{\gamma'} & F_1 \ar[l]_-{\epsilon'} 
}
\end{equation}
of abelian groups, 
then we say that $\Ksix(\A,\I)$ is \emph{isomorphic} to $\mathcal{E}$ 
if there exist isomorphisms $\alpha_i$, $\beta_i$ for $i=1,2,3$ making 
\begin{eqnarray}\label{isodef}
\xymatrix{ K_0(\I) \ar[rr]^-{\iota_*} \ar[rd]_-{\alpha_1} & & 
K_0(\A) \ar[rr]^-{\pi_*} \ar[d]_-{\alpha_2} & & 
K_0(\A/\I) \ar[ddd]^-{\partial_0} \ar[dl]^-{\alpha_3} \\ 
& {G_1}\ar[r]^-{\epsilon} & {G_2}\ar[r]^-{\gamma} 
& {G_3}\ar[d]^-{\delta_0} & \\ 
& {F_3}\ar[u]^-{\delta_1} & {F_2} \ar[l]_-{\gamma'} 
& {F_1}\ar[l]_-{\epsilon'} & \\ 
K_1(\A/\I) \ar[uuu]^-{\partial_1} \ar[ru]^-{\beta_3} & & 
K_1(\A) \ar[ll]_-{\pi_*} \ar[u]^-{\beta_2} & &
K_1(\I) \ar[ll]_-{\iota_*} \ar[ul]_-{\beta_1}}
\end{eqnarray}
commute.  We see from Lemma \ref{trivial-lemma} and functoriality of $K$-theory that when two isomorphic $C^*$-algebras with smallest or largest ideals are given, the associated cyclic six-term exact sequences are isomorphic.

\subsection{Ordered $K$-theory preliminaries}

Lemma~\ref{trivial-lemma} shows that $\Ksix(\A,\I)$ is 
an invariant of a $C^*$-algebra $\A$ 
in the case the ideal $\I$ is either smallest or largest (or both). 
However to get a finer invariant 
we need to consider a pre-order on $K_0$-groups. 

A \emph{pre-ordered abelian group} is a pair $(G, G^+)$, where $G$ is an abelian group and $G^+$ is a subset of $G$ satisfying $G^+ + G^+ \subseteq G^+$ and $0 \in G^+$. For $x, y \in G$ we write $x \leq y$ to mean $y-x \in G^+$.
A group homomorphism $h : G_1 \to G_2$ between pre-ordered abelian groups is called an \emph{order homomorphism} if $h(G_1^+) \subseteq G_2^+$.  If $h$ is also a group isomorphism with  $h(G_1^+) = G_2^+$, then we call $h$ an \emph{order isomorphism}.

If $\A$ is a $C^*$-algebra, then $K_0(\A)$ is a pre-ordered abelian group 
with $K_0(\A)^+:= \{ [p]_0 : p \in M_\infty(\A) \}$.  
If $\A$ contains an approximate unit consisting of projections 
(which is the case for a graph $C^*$-algebra), 
then $K_0(\A)^+ - K_0(\A)^+ = K_0(\A)$. 

For a $C^*$-algebra $\A$ and its ideal $\I$, 
we let $\Ksixo(\A,\I)$ denote the same sequence as \eqref{Ksix} 
but the three $K_0$-groups are considered as pre-ordered abelian groups. 
If we have a cyclic six-term exact sequence $\mathcal{E}^+$ 
of the same form as \eqref{six-term} but  
$G_1$, $G_2$, and $G_3$ are pre-ordered abelian groups, 
then we say that $\Ksixo(\A,\I)$ is \emph{order isomorphic} to $\mathcal{E}^+$  
if $\alpha_i$ is order an isomorphism of pre-ordered groups for $i= 1,2,3$.

A \emph{Riesz group} is an ordered abelian group $G$ that is unperforated (i.e., if $g \in G$, $n \in \N$, and $ng \in G^+$ then $g \in G^+$) and has the Riesz interpolation property (i.e., for all $g_1, g_2, h_1, h_2 \in G$ with $g_i \leq h_j$ for $i,j = 1,2$ there is an element $z \in G$ such that $g_i \leq z \leq h_j$ for $i,j =1,2$).  If $\A$ is an AF-algebra, then $K_0(\A)$ is a countable Riesz group.  Moreover, there is a lattice bijection between the ideals of $\A$ and the ideals of $K_0(\A)$ given by $\I \mapsto \iota_*(K_0(\I))$, 
where $\iota\colon \I \to \A$ is the inclusion map.
If $G$ is an ordered abelian group, then an \emph{ideal} in $G$ is a subgroup $H$ of $G$ with $H^+ = H \cap G^+$, $H = H^+ - H^+$, and whenever $x,y \in G$ with $0 \leq x \leq y$ and $y \in H^+$, then $x \in H$.  An ordered abelian group is \emph{simple} if it has no nontrivial ideals.

A pre-ordered abelian group $G$ is \emph{trivially pre-ordered} if $G^+=G$. 
If $\A$ is a simple purely infinite $C^*$-algebra, 
then $K_0(\A)$ is trivially pre-ordered.

\subsection{Graph and graph $C^*$-algebra preliminaries}\label{graph-subsec}

A (directed) graph $E=(E^0, E^1, r, s)$ consists of a countable set $E^0$ of vertices, a countable set $E^1$ of edges, and maps $r,s: E^1 \rightarrow E^0$ identifying the range and source of each edge.  A vertex $v \in E^0$ is called a \emph{sink} if $|s^{-1}(v)|=0$, and $v$ is called an \emph{infinite emitter} if $|s^{-1}(v)|=\infty$. A graph $E$ is said to be \emph{row-finite} if it has no infinite emitters. If $v$ is either a sink or an infinite emitter, then we call $v$ a \emph{singular vertex}.  We write $E^0_\textnormal{sing}$ for the set of singular vertices.  
Vertices that are not singular vertices are called \emph{regular vertices} and we write $E^0_\textnormal{reg}$ for the set of regular vertices.  
A \emph{cycle} is a sequence of edges $\alpha = \alpha_1 \alpha_2 \ldots \alpha_n$ with $r(\alpha_i) = s(\alpha_{i+1})$ for $1 \leq i < n$ 
and $r(\alpha_n)=s(\alpha_1)$. 
We call the vertex $r(\alpha_n)=s(\alpha_1)$ 
the \emph{base point} of the cycle $\alpha$.  
A \emph{loop} is a cycle of length $1$.

If $E$ is a graph, a \emph{Cuntz-Krieger $E$-family} is a set of mutually orthogonal projections $\{p_v : v \in E^0\}$ and a set of partial isometries $\{s_e : e \in E^1\}$ with orthogonal ranges which satisfy the \emph{Cuntz-Krieger relations}:
\begin{enumerate}
\item $s_e^* s_e = p_{r(e)}$ for every $e \in E^1$;
\item $s_e s_e^* \leq p_{s(e)}$ for every $e \in E^1$;
\item $p_v = \sum_{s(e)=v} s_e s_e^*$ for every $v \in E^0$ that is
not a singular vertex.
\end{enumerate} The \emph{graph $C^*$-algebra $C^*(E)$} is defined to be the $C^*$-algebra generated by a universal Cuntz-Krieger $E$-family.  The graph $C^*$-algebra is unital if and only if $E^0$ is a finite set in which case $1_{C^*(E)} = \sum_{v \in E^0} p_v$.

For any graph $E$, the \emph{regular vertex matrix} 
is the $E^0 \times E^0_\textnormal{reg}$ matrix $R_E$ 
with 
\[
R_E(v,w) := | \{ e \in E^1 :  \text{$r(e)= v$ and $s(e)=w$}  \} |.
\]
Since $w \in E^0_\textnormal{reg}$, 
all entries of $R_E$ are finite, and $R_E$ is column-finite. 
Hence we get $R_E \in M_{E^0, E^0_\textnormal{reg}}(\Z)$. 
We note that by the definition of regular vertices 
each column contains at least one nonzero entry. 
If $E$ has no regular vertices then 
$R_E = \emptyset \in M_{E^0, \emptyset}(\Z)$. 
Recall that $I \in M_{E^0, E^0_\textnormal{reg}}(\Z)$ 
is defined by $I(\delta_v)=\delta_v$ for $v \in E^0_\textnormal{reg}$.

\begin{proposition}[{\cite[Theorem~3.2]{RS}, \cite[Theorem~3.1]{DT2}}]
\label{K-group-computation-prop}
Let $E$ be a graph. 
Then we have 
\[
K_0(C^*(E)) \cong \coker (R_E-I)
\quad \text{ and } \quad
K_1(C^*(E)) \cong \ker (R_E-I). 
\]
\end{proposition}

From this proposition, 
we see that $K_0(C^*(E))$ and $K_1(C^*(E))$ are countable abelian groups, 
and in addition, $K_1(C^*(E))$ is a free abelian group 
because any subgroup of a free abelian group is free.  
We also have 
\begin{equation}\label{rank-computation}
\rank K_0(C^*(E)) + |E^0_\textnormal{reg}|
=\rank K_1(C^*(E)) + |E^0|. 
\end{equation}

Let $E$ be a graph. 
A subset $H \subseteq E^0$ is \emph{hereditary} 
if whenever $e \in E^1$ and $s(e) \in H$, then $r(e) \in H$. 
A hereditary subset $H$ is \emph{saturated} 
if whenever $v \in E^0_\textnormal{reg}$ 
with $r(s^{-1}(v)) \subseteq H$, then $v \in H$. 
For a hereditary subset $H$, 
we can define two graphs $F$ and $G$ by 
\begin{align} \label{two-subgraphs}
\myF&=(H,s^{-1}(H),r,s),& 
\myG&=(E^0\setminus H, E^1 \setminus r^{-1}(H),r,s)
\end{align}
where $r$ and $s$ are restrictions of those for $E$. 
The set $H$ is saturated 
if and only if we have 
$\myGp^0_\textnormal{reg} \supseteq E^0_\textnormal{reg} \setminus H$. 
If a saturated hereditary subset $H$ 
satisfies 
$\myGp^0_\textnormal{reg}=E^0_\textnormal{reg} \setminus H$
then we say that $H$ has \emph{no breaking vertices}. 
Note that if $E$ is row-finite, then 
every saturated hereditary subset has no breaking vertices.

For a saturated hereditary subset $H$, 
we denote by $\I_H$ the ideal of $C^*(E)$ generated
by $\{ p_v : v \in H \}$.  

\begin{proposition} \label{Ksix-from-snake-prop}
Let $E = (E^0, E^1, r, s)$ be a graph, 
and let $H$ be a saturated hereditary subset of $E^0$ 
such that $H$ has no breaking vertices.  
Let $\myF$ and $\myG$ be the two graphs as in \eqref{two-subgraphs} for $H$. 
Then we have the following: 
\begin{enumerate}
\item There is a natural embedding from $C^*(\myF)$ 
onto a full corner of $\I_H$.
\item We have $C^*(E)/\I_H \cong C^*(\myG)$.
\item We have  $E^0 = \myF^0 \sqcup \myG^0$ and $E^0_\textnormal{reg} = \myFp^0_\textnormal{reg} \sqcup \myGp^0_\textnormal{reg}$.
\item There exists a row-finite matrix 
$X \in M_{\myF^0, \myGp^0_\textnormal{reg}}(\Z^+)$ 
such that under the decomposition in (3), we get 
$R_E =\left( \begin{smallmatrix} R_{\myF} & X \\ 0 & R_{\myG} \end{smallmatrix} \right).$
\item $\Ksix (C^*(E), \I_H)$ is isomorphic to 
\begin{equation*}
\xymatrix{ \coker (R_{\myF}-I) \ar[r] & \coker {\left( \begin{smallmatrix} R_{\myF}-I & X \\ 0 & R_{\myG}-I \end{smallmatrix} \right)} \ar[r] & \coker (R_{\myG}-I)\phantom{.} \ar[d]^0 \\
\ker (R_{\myG}-I) \ar[u]^{[X]} &  \ker  {\left( \begin{smallmatrix} R_{\myF}-I & X \\ 0 & R_{\myG}-I \end{smallmatrix} \right)} \ar[l] & \ker (R_{\myF}-I) \ar[l]
}
\end{equation*}
where the horizontal maps are the obvious inclusions or projections, 
and $[X]$ is the map implemented by multiplication by $X$.  
\end{enumerate}
\end{proposition}

\begin{proof}
Statements (1) and (2) are standard (see \cite{CET}). 
It is straightforward to check (3) and (4). 
Finally (5) follows from \cite[Remark~4.2]{CET} and \cite[Theorem~1]{CET}. 
\end{proof}

\begin{remark}
One can show that the sequence (5) above is nothing but 
the long exact sequence obtained by applying the Snake Lemma 
from homological algebra (see \cite{sm:h}) to 
the commutative diagram 
\begin{equation*}
\xymatrix{
0 \ar[r] & \Z^{\myFp^0_\textnormal{reg}} \ar[r] \ar[d]_{R_{\myF}-I} & \Z^{E^0_\textnormal{reg}} \ar[r] \ar[d]_{R_{E}-I} & \Z^{\myGp^0_\textnormal{reg}} \ar[r] \ar[d]_{R_{\myG}-I} & 0 \\
0 \ar[r] & \Z^{\myF^0} \ar[r] & \Z^{E^0} \ar[r] & \Z^{\myG^0} \ar[r] & 0
}
\end{equation*}
with exact rows. 
See Proposition~\ref{Snake-prop}. 
\end{remark}

\begin{remark}\label{gaugevanishes}
For a graph $E$ and a gauge-invariant ideal $\I$ of $C^*(E)$, 
there is a computation of $\Ksix (C^*(E), \I)$ in \cite[4.1]{CET}
similar to (5) of Proposition~\ref{Ksix-from-snake-prop}. 
In particular, the index map from $K_0(C^*(E)/\I)$ to $K_1(\I)$ is always $0$ in this case. 
\end{remark}

\begin{remark}
If $C^*(E)$ has a unique nontrivial ideal $\I$, 
then $\I=\I_H$ for a saturated hereditary subset $H$ of $E^0$ 
such that $H$ has no breaking vertices \cite[Lemma 3.1]{ET1}. 
\end{remark}

%%%%%%%%%%%%%%%%%%%%%%%%%%%%%%%%%%%%%%%%%%%%%%%%%%%%%%
\section{$K$-groups of AF graph $C^*$-algebras and simple graph $C^*$-algebras} \label{simple-realize-sec}
%%%%%%%%%%%%%%%%%%%%%%%%%%%%%%%%%%%%%%%%%%%%%%%%%%%%%%

It has been shown by the work of many hands that 
$K$-theoretic invariants completely classify, up to stable isomorphism,
the class of AF graph $C^*$-algebras 
and the class of simple graph $C^*$-algebras. 
The range of the invariants has also been computed. 
We review and reprove some of these results in this section to get sharper results regarding realization of graphs. 

The following is a refined realization of AF graph $C^*$-algebras. 

\begin{proposition}  \label{realize-simple-K-theories-AF-prop}
If $(G, G^+)$ is a countable Riesz group, then there exists a row-finite graph $E$ such that
\begin{enumerate}[(1)]
\item $E$ has no sinks, no sources, and a countably infinite number of vertices,
\item $E$ has no cycles (so that, in particular, $C^*(E)$ is an AF-algebra), 
\item $(K_0(C^*(E)), K_0(C^*(E))^+) \cong (G, G^+)$, and
\item  $C^*(E)$ is stable
\end{enumerate}
\end{proposition}

\begin{proof}
By the Effros-Handelman-Shen Theorem \cite{EHS}, there exists an AF-algebra $A$ whose $K_0$-group is order isomorphic to $(G, G^+)$.  By Drinen's Theorem \cite[Theorem~1]{Dri} there exists a row-finite graph $F$ such that $F$ has no cycles and $C^*(F)$ is stably isomorphic to $A$.  Let $E$ be the graph obtained by adding a tail (cf.\ \cite{DT1}) to every sink of $F$ and a head to every source of $F$ (cf.\ \cite{DT1}).  Then $E$ has no sinks or sources, $E$ has a countably infinite number of vertices, $E$ has no cycles, and $C^*(E)$ is isomorphic to $C^*(F)\otimes\K$ so that $(K_0(C^*(E)), K_0(C^*(E))^+) \cong (G, G^+)$. 
\end{proof}

We next consider simple graph $C^*$-algebras. 
The ``Dichotomy for simple graph $C^*$-algebras" \cite[Remark~2.16]{DT1} states that any simple graph $C^*$-algebra $C^*(E)$ is either purely infinite (if $E$ contains a cycle) or AF (if $E$ contains no cycles).  Consequently, the Kirchberg-Phillips Classification Theorem and Elliott's Theorem imply that any simple graph $C^*$-algebra is classified up to stable isomorphism by the pair $(K_0(C^*(E)), K_1(C^*(E)))$, where we consider $K_0(C^*(E))$ as a pre-ordered abelian group.  
If $C^*(E)$ is purely infinite, then $K_0(C^*(E))$ is trivially pre-ordered; i.e.\ $K_0(C^*(E))^+= K_0(C^*(E))$.  If $C^*(E)$ is AF, then $K_0(C^*(E))$ is an ordered group (in fact, a Riesz group) and $K_0(C^*(E))^+$ is a proper subset of $K_0(C^*(E))$.  Thus the ordering on the $K_0$-group can distinguish whether the simple graph $C^*$-algebra is purely infinite or AF.
We have already seen in Proposition~\ref{realize-simple-K-theories-AF-prop} 
that all simple Riesz groups are realized as the $K_0$-group 
of an AF graph $C^*$-algebra (and, moreover, that the graph may be chosen to have certain properties). 
Since an AF-algebra whose $K_0$-group is a simple Riesz group is simple, 
AF graph $C^*$-algebras given 
in Proposition~\ref{realize-simple-K-theories-AF-prop} 
for simple Riesz groups are necessarily simple. 
So in order to complete the description of the range 
of $K$-theoretic invariants for simple graph $C^*$-algebras, 
we only need to consider simple purely infinite graph $C^*$-algebras. 
We know that the pre-order of the $K_0$-group has to be trivial. 
We also know that the $K_1$-group has to be free 
by Proposition~\ref{K-group-computation-prop}.   Szyma\'nski proved that these are the only restrictions on the $K$-groups, and we prove a sharper version of his result below (see Proposition~\ref{realize-simple-K-theories-prop}) that gives us extra control of the choice of a graph.

\begin{lemma} \label{exact-sequence-exists-lem}
Let $G$ be a countable abelian group and let $F$ be a countable free abelian group.  Then there exists an exact sequence
$$
\xymatrix{
0 \ar[r] & F \ar[r]^{\iota} & \Z^\infty \ar[r]^{\phi} & \Z^\infty \ar[r]^{\pi} & G \ar[r] & 0.
}
$$
\end{lemma}

\begin{proof}
Let $X$ be a generating set for $G$, which must be countable since $G$ is countable. Let $\mathcal{F}(X)$ be the free abelian group generated by $X$. Then $\mathcal{F}(X) \cong \Z^{m}$, where $m=|X|$, and by the universal property of free abelian groups there exists a surjective homomorphism $\pi_0 : \Z^{m}\to G$.  Define $\pi : \Z^\infty \oplus \Z^m \to G$ by $\pi (x,y) := \pi_0(y)$.  Since $\ker \pi_0$ is a subgroup of a countable free abelian group, $\ker \pi$ is a countable free abelian group.  Moreover, since $\Z^\infty \oplus 0 \subseteq \ker \pi$, it follows that there is an isomorphism $\psi : \Z^\infty \to \ker \pi$.  In addition, since $F$ is a countable free abelian group there exists an isomorphism $i_0 : F \to \Z^n$ for some $n \in \{0, 1, 2, \ldots, \infty \}$.  Define $\phi : \Z^\infty \oplus \Z^n \to \Z^\infty \oplus \Z^m$ by $\phi(x,y) = \psi(x)$, and define $i: F \to \Z^\infty \oplus \Z^n$ by $i(x) := (0, i_0(x))$.  One can verify that 
$$
\xymatrix{
0 \ar[r] & F \ar[r]^<>(.5){i} & \Z^\infty \oplus \Z^n \ar[r]^{\phi} & \Z^\infty \oplus \Z^m \ar[r]^<>(.5){\pi} & G \ar[r] & 0
}
$$
is exact.  Since $Z^\infty \oplus \Z^n \cong \Z^\infty$ and $Z^\infty \oplus \Z^m \cong \Z^\infty$, the result follows. 
\end{proof}

\begin{proposition}[Szyma\'nski's Theorem] 
\label{realize-simple-K-theories-prop}
Let $G$ be a countable abelian group and let $F$ be a countable free abelian group.  Then there exists a row-finite graph $E$ 
with countably infinite $E^0$ 
that satisfies the following properties:
\begin{enumerate}[(1)]
\item Every vertex in $E$ is the base point of at least two loops, 
\item $E$ is transitive (so that, in particular, $C^*(E)$ is simple and purely infinite),
\item $K_0(C^*(E)) \cong G$ and $K_1(C^*(E)) \cong F$, and
\item $C^*(E)$ is stable.
\end{enumerate}
\end{proposition}

\begin{proof}
By Lemma~\ref{exact-sequence-exists-lem} there exists a group homomorphism $\phi : \Z^\infty \to \Z^\infty$ with $\coker \phi \cong G$ and $\ker \phi \cong F$.  Let $A_0 \in M_\infty (\Z)$ be the matrix representation of $\phi$.  Then $A_0$ is a column-finite matrix.  Define $|A_0| \in M_\infty (\Z)$ to be the entry-wise absolute value of $A_0$; i.e., $|A_0| (i,j) := |A_0(i,j)|$.  
Also define 
\[
A_1 := |A_0| + \begin{pmatrix} 
1 & 1 & 0 & 0 &  \cdots \\
1 & 1 & 1 & 0 &  \cdots \\
0 & 1 & 1 & 1 &  \\
0 & 0 & 1 & 1 &  \\
\vdots & \vdots &  & & \ddots 
\end{pmatrix}. 
\]
Let $\I$ denote the identity matrix in $M_\infty (\Z)$, and define 
\[
A := \begin{pmatrix} A_0 + A_1 + I & A_1 \\ I & 2 I \end{pmatrix}.
\]
We observe that $A$ is a column finite square  matrix with non-negative entries, and also observe that the diagonal is everywhere greater or equal to 2.
Hence we can find a graph $E$ with no singular vertices 
such that $R_E=A$. 
Since $A$ is indexed by a countably infinite set, 
the set $E^0$ of vertices is countable and infinite. 
Since every vertex in $E$ is regular, $E$ is row-finite. 
In addition, the fact that every diagonal entry of $A$ is two or larger 
shows that every vertex in $E$ is the base point of at least two loops.  
Furthermore, one can see from the definition of $A$ that $E$ is transitive.  
Finally, \nobox{we have} 
\[
A-I = \begin{pmatrix} A_0 + A_1 & A_1 \\ I & I \end{pmatrix}
=\begin{pmatrix} I & A_1 \\ 0 & I \end{pmatrix}
\begin{pmatrix} A_0 & 0 \\ 0 & I \end{pmatrix}
\begin{pmatrix} I & 0 \\ I & I \end{pmatrix}.
\]
Since the matrices 
$\left( \begin{smallmatrix} I & A_1 \\ 0 & I \end{smallmatrix} \right)$ and 
$\left( \begin{smallmatrix} I & 0 \\ I & I \end{smallmatrix} \right)$ 
are invertible (with the inverses 
$\left( \begin{smallmatrix} I & -A_1 \\ 0 & I \end{smallmatrix} \right)$ and 
$\left( \begin{smallmatrix} I & 0 \\ -I & I \end{smallmatrix} \right)$ ), 
$A-I$ has a kernel and cokernel which is isomorphic to those of  the matrix $\left( \begin{smallmatrix} A_0  & 0 \\ 0 & I \end{smallmatrix} \right)$, 
\nobox{and hence} to those of  $A_0$. 
Thus
$$K_0(C^*(E)) \cong \coker (A-I) \cong \coker A_0 \cong G$$ and $$K_1(C^*(E)) \cong \ker (A-I) \cong \ker A_0 \cong F.$$
Finally, since 
$C^*(E)$ is a nonunital Kirchberg algebra, it is stable by \cite{zhangdicho}.
\end{proof}

\begin{remark}
Proposition~\ref{realize-simple-K-theories-prop}, together with Lemma~\ref{exact-sequence-exists-lem}, gives a shorter proof of Szyma\'nski's Theorem \cite[Theorem~1.2]{Szy2} \nobox{with} the added conclusion (1).
\end{remark}

The following proposition summarizes the arguments above. 

\begin{proposition}
\label{realize-simple-AF-K-theories-prop}
The class of graph $C^*$-algebras that are either AF or simple 
is classified up to stable isomorphism 
by $K_0$-groups as pre-ordered abelian groups 
and $K_1$-groups as abelian groups. 
A pair $(G,F)$ of a countable pre-ordered abelian group $G$ 
and a countable abelian group $F$ 
is in the range of invariants in this class if and only if either 
\begin{itemize}
\item $G$ is a Riesz group and $F=0$, or 
\item $G$ is trivially preordered and $F$ is free. 
\end{itemize}
Moreover, one can realize the above invariants 
by a stable graph $C^*$-algebra $C^*(E)$ for 
a row-finite graph $E$ with no sinks and no sources. 
\end{proposition}

In order to get a finer invariant for 
the classification up to isomorphism, 
one needs to consider scales of $K_0$-groups in the AF case, 
and positions of units in $K_0$-groups 
in unital simple purely infinite case. 
With this extra data, one can classify, up to isomorphism,  
the  all graph $C^*$-algebras that are either AF or simple.
However, the computation of the range of this finer invariant 
is not as straightforward as above. 
In fact, the (nonunital, but nonstable) case of AF graph $C^*$-algebras 
is very complicated (see \cite{tkasmt:ragaelaua}).
For the nonunital simple purely infinite case, 
$K_0$-groups and $K_1$-groups are already 
a complete invariant up to isomorphism 
since they are stable. 
In what follows, 
we complete the computation of the range of the invariant 
for the unital simple purely infinite case. 

For a unital simple purely infinite $C^*$-algebra $A$, 
the extra information we need for classification 
up to isomorphism is the element $[1_A]_0 \in K_0(A)$ 
defined by the unit $1_A$ of $A$. 
For a group $G$ and an element $g_0$, 
we write $(K_0(A),[1_A]_0) \cong (G,g_0)$ 
if there exists an isomorphism from $K_0(A)$ to $G$ 
sending $[1_A]_0$ to $g_0$. 

Recall that a graph $C^*$-algebra $C^*(E)$ is unital 
if and only if the set $E^0$ of vertices is finite, 
and in this case the unit $1_{C^*(E)}$ is $\sum_{v \in E^0}p_v$. 
Under the isomorphism 
$K_0(C^*(E)) \cong \coker (R_E-I)$ 
in Proposition~\ref{K-group-computation-prop} 
the element $[1_{C^*(E)}]_0$ corresponds to the equivalence class
$[{\mathbf 1}]$ 
where ${\mathbf 1} := \sum_{v \in E^0}\delta_v \in \Z^{E^0}$ 
is the vector all of whose entries are $1$. 

Note that if $E^0$ is finite, 
then Proposition~\ref{K-group-computation-prop} 
and \eqref{rank-computation} 
shows that the two groups $G := K_0(C^*(E))$ and $F := K_1(C^*(E))$ 
satisfy 
\begin{itemize}
\item $G$ is a finitely generated abelian group, 
\item $F$ is a finitely generated free abelian group 
with $\rank F \leq \rank G$.  
\end{itemize}
The following proposition shows that 
these are the only restrictions for the $K$-groups, 
and there is no restriction on the position of $[1_{C^*(E)}]$ 
in $K_0(C^*(E))$. 
In this case, as opposed to what we saw above, not every vertex may be chosen regular.

\begin{proposition} \label{realize-unital-algebras}
Let $G$ be a finitely generated abelian group, and let $F$ be a free abelian group with $\rank F \leq \rank G$.  
Let $g_0$ be an element of $G$. 

Then there exists a graph $E$ with finite $E^0$ 
that satisfies (1)--(2) of Proposition \ref{realize-simple-AF-K-theories-prop} as well as
\begin{enumerate}[(1')]\addtocounter{enumi}{2}
\item $(K_0(C^*(E)),[1_{C^*(E)}]_0) \cong (G,g_0)$ 
and $K_1(C^*(E)) \cong F$.
\end{enumerate}
\end{proposition}

\begin{proof}
By the fundamental theorem of finitely generated abelian groups,
there exist unique integers $k, n \geq 0$ 
and $m_1,m_2, \ldots, m_k \geq 2$ 
with $m_1 \mid m_2 \mid \cdots \mid m_k$
such that 
\begin{equation}\label{groupform}
G \cong 
(\Z/m_1\Z) \oplus (\Z/m_2\Z) \oplus \cdots \oplus (\Z/m_k\Z) \oplus \Z^n. 
\end{equation}
We can then find a generating set $\{\gamma_i\}_{i=1}^{k+n}$ of $G$ 
with the relations $m_i \gamma_i =0$ for $i=1,\ldots,k$. 
There exist integers $(a_i)_{i=1}^{k+n}$ 
such that $g_0=\sum_{i=1}^{k+n} a_i \gamma_i$. 
By subtracting $m_i$ from $a_i$ for $i\leq k$ many times 
and replacing $\gamma_i$ with $-\gamma_i$ for $i>k$ if necessary, 
we may assume that $a_i \leq 0$ for all $i$. 
Let $n' = \rank F$, which is at most $\rank G =n$ by the assumption. 

We denote the bases of $\Z^{1+k+n}$ and $\Z^{1+k+n'}$ 
by $\{\delta_i\}_{i=0}^{k+n}$ and $\{\delta_i\}_{i=0}^{k+n'}$ 
(the indices have been shifted compared to the convention in Section~\ref{prelim-sec}). 
We define a surjective map $\pi_0\colon \Z^{1+k+n} \to G$ 
by $\pi_0(\delta_0)=0$ and $\pi_0(\delta_i)=\gamma_i$ 
for $i=1,2,\ldots,k+n$, and we define $A_0 \in M_{1+k+n,1+k+n'}(\Z)$ 
by $A_0(\delta_0)=\delta_0$, $A_0(\delta_i)=m_i\delta_i$, 
for $i=1,\ldots,k$ and $A_0(\delta_i)=0$ 
for $i=k+1,\ldots,k+n'$. 
Then we have $\im A_0 = \ker \pi_0$. 
In matrix form, we have 
\[
A_0=
\left(\begin{smallmatrix}
1 & & & &   & & \\
 & m_1 & & \nobigzerol &  & \nobigzerol & \\
 & & m_2 & &  &  & \\
 & \nobigzerol & & \ddots &   & & \\
 & & & & m_k  & & \\ 
& \nobigzerou & & &   &0 &\\
&&&&&&\ddots 
\end{smallmatrix}\right).
\]

We set $b_i=1-a_i \geq 1$ for $i=1,2,\ldots,k+n$. 
Consider two square matrices $P \in M_{1+k+n,1+k+n}(\Z)$ 
and $Q \in M_{1+k+n',1+k+n'}(\Z)$ by 
\begin{align*}
P&=
\left(\begin{smallmatrix}
1 & & & &   & & \\
b_1 & 1 & & &  & & \\
b_2 & & 1 & &  & \nobigzerol & \\
\vdots & & & \ddots &   & & \\
b_k & & & & 1  & & \\ 
\vdots & & \nobigzerou & &   & \ddots& \\
b_{k+n}& & & &   & & 1
\end{smallmatrix}\right),&
Q&=
\left(\begin{smallmatrix}
1 & 1& \cdots & 1  & \cdots & 1\\
 & 1 & &  & & \\
 & & \ddots &   & \nobigzerol & \\
 & & & 1  & & \\ 
 & & \nobigzerou &   & \ddots& \\
 & & &   & & 1
\end{smallmatrix}\right).
\end{align*}
Note that $P$ and $Q$ are invertible. 
We set $A \in M_{1+k+n,1+k+n'}(\Z)$ by 
\[
A=PA_0Q=
\begin{pmatrix}
1 & 1& 1& \ldots & 1& \ldots & 1\\
b_1 & b_1+m_1 & b_1& \ldots & b_1 & \ldots & b_1\\
b_2 & b_2 & b_2+m_2 & & b_2 & \ldots & b_2 \\
\vdots  &\vdots & & \ddots &   & & \vdots\\
b_k & b_k& b_k & & b_k+m_k  & \vdots &b_k \\ 
\vdots& \vdots &\vdots & & \cdots  & & \vdots \\
b_{k+n}& b_{k+n} & b_{k+n} & \cdots & b_{k+n}  & \cdots & b_{k+n}
\end{pmatrix}.
\]
Since $b_i \geq 1$ and $m_i \geq 2$, 
all entries of $A$ are positive. 
We define $\pi\colon \Z^{1+k+n} \to G$ by $\pi = \pi_0\circ P^{-1}$. 
Then $\pi$ is a surjection satisfying 
\[
\ker \pi = P(\ker \pi_0)
= P(\im A_0) = \im(PA_0)=\im A, 
\]
and hence $\pi$ induces an isomorphism 
\[
\bar{\pi}\colon \coker A \ni [x] \mapsto \pi(x) \in G,
\]
and therefore $\ker A$ is a free abelian group whose rank is 
$n+(1+k+n')-(1+k+n)=n'$. 

We set ${\mathbf 1} = \sum_{i=0}^{k+n}\delta_i \in \Z^{1+k+n}$. 
We are going to show $\bar{\pi}([{\mathbf 1}]) = g_0$. 
Since $b_i+a_i=1$ for $i=1,2,\ldots,k+n$, 
we have 
\[
P\Big(\delta_0+\sum_{i=1}^{k+n}a_i\delta_i\Big)
= \sum_{i=0}^{k+n}\delta_i = {\mathbf 1},\quad 
\begin{pmatrix}
1 & & & \\
b_1 & 1 & & \\
\vdots & & \ddots & \\
b_{k+n}& & & 1
\end{pmatrix}
\begin{pmatrix}
1\\
a_1\\
\vdots\\
a_{k+n}
\end{pmatrix}
=
\begin{pmatrix}
1\\
1\\
\vdots\\
1
\end{pmatrix}.
\]
Hence we have 
\begin{align*}
\bar{\pi}([{\mathbf 1}])
&=\pi({\mathbf 1}) 
=\pi_0(P^{-1}({\mathbf 1})) =\pi_0\Big(\delta_0+\sum_{i=1}^{k+n}a_i\delta_i\Big) 
=\sum_{i=1}^{k+n}a_i \gamma_i
=g_0. 
\end{align*}
Let $E$ be the graph such that 
\begin{itemize}
\item $E^0=\{0,1,\ldots,k+n\}$, 
\item there are infinitely many edges from $i \in E^0$ 
with $k+n'<i\leq k+n$ to every vertex,  and
\item $0,1,\ldots,k+n'$ are regular, 
and the regular vertex matrix $R_E$ of $E$ is $A+I$. 
\end{itemize}
Then $E$ is a graph with $1+k+n$ vertices, 
each vertex is the base point of at least two loops, 
and $E$ is transitive.  
The computation of $K$-theory follows from 
Proposition~\ref{K-group-computation-prop} and the first part of this proof. 
\end{proof}

\begin{remark} \label{smallest-vertices-remark}
Take $G,F$ and $g_0 \in G$ as in Proposition~\ref{realize-unital-algebras}, 
and $k,n,m_1,\ldots,m_k$ as in \eqref{groupform}.
Below we will show that 
if a graph $E$ satisfies the condition (3') 
in Proposition~\ref{realize-unital-algebras} 
then the number $|E^0|$ of vertices of $E$ is at least $k+n$. 
We also show that in many cases 
including the case that $G,F$ are arbitrary, but $g_0=0$, 
$|E^0|$ needs to be bigger than $k+n$. 
Thus the number $1+k+n$ of vertices of the graph $E$ 
in the proof above is smallest possible in these cases. 

Observe that if $E$ satisfies the condition (3) 
then we have a surjective map $\Z^{E^0} \to G$ 
sending ${\mathbf 1}$ to $g_0$. 
Choose a prime number $p$ with $p \mid m_1$. 
Tensoring with $\Z / p\Z$, 
we get a surjective map 
\[
(\Z / p\Z)^{E^0} \cong \Z^{E^0} \otimes (\Z / p\Z) 
\to G\otimes (\Z / p\Z) \cong (\Z / p\Z)^{k+n}
\]
This shows that $|E^0| \geq k+n$. 
Since ${\mathbf 1}\otimes 1$ is a nonzero element of $(\Z / p\Z)^{E^0}$, 
if $g_0 \otimes 1 \in G\otimes (\Z / p\Z)$ is zero, 
then the above surjection is not injective. 
Thus in this case we have $|E^0| > k+n$. 

We also note that for arbitrary $G,F$, 
there exists $g_0 \in G$ such that 
we can find a graph $E$ satisfying 
the condition (3) with $|E^0| = k+n$.
However, in the case $G$ is free and $F \neq 0$, 
or in the case $G=F=0$, 
there exists no such $E$ with $C^*(E)$ is simple. 
Thus for such $G$ and $F$ and for arbitrary $g_0 \in G$, 
the graph $E$ constructed in the proof of 
Proposition~\ref{realize-unital-algebras} 
has the smallest possible number of vertices, namely $|E^0|=1+k+n$, 
such that $C^*(E)$ is simple and satisfies condition (3). 
For a pair $(G,F)$ other than the ones mentioned above, 
there exists some $g_0 \in G$ (necessarily nonzero) 
such that 
we can find a graph $E$ with $|E^0|=k+n$ 
satisfying the all three conditions in 
Proposition~\ref{realize-unital-algebras}. 
\end{remark}

We need the next small variation of 
Proposition~\ref{realize-unital-algebras} 
in order to control the unit in the extension. 
As explained in Remark~\ref{smallest-vertices-remark}, 
if $G,F$ and $g_0 \in G$ satisfies 
either $g_0=0$ or $G$ is free (and $F$ is nonzero) 
then a graph $E$ as in Proposition~\ref{realize-unital-algebras}
has at least $1+k+n$ vertices. 
If $G$ satisfies \emph{both} of the two conditions, 
then we need one more vertex to get the next result. 

\begin{proposition} \label{realize-unital-algebras-2}
Let $G,F$ and $g_0 \in G$ 
be as in Proposition~\ref{realize-unital-algebras}. 
Then there exists a graph $E$ with finite $E^0$
satisfying (1)--(3) in Proposition~\ref{realize-unital-algebras}, as well as
\begin{itemize}
\item[(4)] there exist two vertices $v, w \in E^0$ 
such that $(R_E-I)(w,v')<(R_E-I)(v,v')$ 
for all $v' \in E^0_\textnormal{reg}$.
\end{itemize}
When $G$ is written in the form of \eqref{groupform} we may choose $E$ with $|E^0|=1+k+n$.
\end{proposition}

\begin{proof}
First consider the case $g_0\neq 0$. 
In this case we show that the graph $E$ constructed in the proof of 
Proposition~\ref{realize-unital-algebras} satisfies (4). 
For $(a_i)_{i=1}^{k+n}$ as in that proof, 
there exists $i \in \{1,2,\ldots,k+n\}$ with $a_i < 0$. 
Then we have $b_i \geq 2$. 
Hence $A(0,j) = 1 < b_i \leq A(i,j)$ for all $j = 0, 1, \ldots, k+n'$. 
Since $R_E-I = A$ for the graph $E$, 
the vertices $v=i$ and $w=0$ satisfy (4). 

Next consider the case $G$ is not free. 
Then we can choose $(a_i)_{i=1}^{k+n}$ 
in the proof of Proposition~\ref{realize-unital-algebras} 
such that $a_1 \leq -m_1 < 0$. 
As in the case $g_0\neq 0$, 
the vertices $v=1$ and $w=0$ satisfy (4). 

Finally suppose $g_0=0$ and $G$ is free. 
Let $n$ and $n'$ be the ranks of $G$ and $F$ respectively. 
We define a $(n+2)\times (n'+2)$ matrix $A$ by 
\[
A:=
\begin{pmatrix}
3 & 2 & \ldots & 2\\
2 & 1 & \ldots & 1 \\
2 & 1 & \ldots & 1 \\
\vdots & \vdots  & \ldots & \vdots\\
2 & 1 & \ldots & 1 
\end{pmatrix}, 
\]
whose image is generated by two elements 
\[
{\mathbf 1} := 
\begin{pmatrix}
1 \\ 1 \\ 1 \\ \vdots \\ 1
\end{pmatrix}
=A
\begin{pmatrix}
1 \\ -1 \\ 
0\vspace{-0.2cm} \\ \vdots\vspace{-0.1cm} \\ 0
\end{pmatrix}	 \quad \text{and}\quad 
\begin{pmatrix}
1 \\ 0 \\ 0 \\ \vdots \\ 0
\end{pmatrix}
=A
\begin{pmatrix}
-1 \\ 2 \\ 
0\vspace{-0.2cm} \\ \vdots\vspace{-0.1cm} \\ 0
\end{pmatrix}.
\]
Hence $\ker A \cong \Z^{n'}$, 
$\coker A \cong \Z^{n}$ and 
$[{\mathbf 1}]=0$ in $\coker A$. 
We define a graph $E$ so that 
\begin{itemize}
\item $E^0=\{1,2,\ldots,n+2\}$, 
\item there are infinitely many edges from $i \in E^0$ 
with $n'+2< i \leq n+2$ to every vertex, 
\item $1,2,\ldots,n'+2$ are regular, 
and the regular vertex matrix $R_E$ of $E$ is $A+I$. 
\end{itemize}
Then $E$ satisfies the four conditions in 
Proposition~\ref{realize-unital-algebras}. 
Finally, $v=1$ and $w=2$ satisfy (4). 
\end{proof}

A Cuntz-Krieger algebra ${\mathcal O}_A$ is 
isomorphic to the $C^*$-algebra of a finite graph 
with no sinks or sources. 
For such a graph $E$, every vertex is regular. 
Therefore, 
Proposition~\ref{K-group-computation-prop} 
and \eqref{rank-computation} 
shows that the $K$-groups $G = K_0({\mathcal O}_A)$ and $F = K_1({\mathcal O}_A)$ 
of a Cuntz-Krieger algebra ${\mathcal O}_A$ satisfy 
\begin{itemize}
\item $G$ is a finitely generated abelian group, 
\item $F$ is a finitely generated free abelian group 
with $\rank F = \rank G$.  
\end{itemize}
The following proposition shows that 
these are the only restrictions for the $K$-groups 
of simple Cuntz-Krieger algebras, 
and there is no restriction on the position of $[1_{{\mathcal O}_A}]_0$ 
in $K_0({\mathcal O}_A)$. 

\begin{proposition} \label{realize-CK-algebras}
Let $G$ be a finitely generated abelian group, and let $F$ be a finitely generated free abelian group with $\rank F = \rank G$.  
Let $g_0$ be an element of $G$. 

Then there exists a graph $E$ with finite vertex set $E^0$ and finite edge set $E^1$ 
that satisfies properties (1)--(4) of Proposition \ref{realize-unital-algebras-2}.
With $G$ on the form \eqref{groupform} we may choose $E$ with $|E^0|=1+k+n$.
\end{proposition}

\begin{proof}
The graph $E$ constructed in the proof of 
Proposition~\ref{realize-unital-algebras} 
or Proposition~\ref{realize-unital-algebras-2}
has finitely many edges because the rank $n'$ of $F$ 
coincides with the rank $n$ of $G$. 
Hence this graph $E$ has the desired properties. 
\end{proof}

%%%%%%%%%%%%%%%%%%%%%%%%%%%%%%%%%%%%%%%%%%%%%%%%%%%%%%
\section{A method for realizing six-term exact sequences} \label{comm-diagram-sec}
%%%%%%%%%%%%%%%%%%%%%%%%%%%%%%%%%%%%%%%%%%%%%%%%%%%%%%

In Proposition~\ref{piece-together-prop} of this section we prove a result that will allow us to 
realize a given six-term exact sequence with four groups already given as kernels and cokernels of certain matrices $A$ and $B$ of a certain form determined by a block matrix $\left(\begin{smallmatrix}A&Y\\0&B
  \end{smallmatrix}\right)$, just as in (5) of Proposition~\ref{Ksix-from-snake-prop}. 
 We start out by noting that in fact every such sequence has this form, provided only that the relevant index map vanishes:

\begin{proposition} \label{Snake-prop}
Let us take $A \in M_{n_1, n_1'} (\Z)$ 
and $B \in M_{n_3, n_3'} (\Z)$ 
for some $n_1, n_1', n_3, n_3' \in \{0,1,2,\ldots,\infty\}$. 
Then 
$Y \mapsto \left(\begin{smallmatrix} A & Y \\ 0 & B \end{smallmatrix} \right)$
is a bijection from 
$M_{n_3', n_1} (\Z)$ to the set of all matrices 
$X \in M_{n_1+n_3, n_1'+n_3'} (\Z)$ and the diagram 
\[
\xymatrix{
0 \ar[r]& \Z^{n_1'} \ar[r]^-{I'} \ar[d]^-{A} 
& \Z^{n_1'}\oplus\Z^{n_3'} \ar[r]^-{P'} \ar@{.>}[d]^-{X} 
&\Z^{n_3'} \ar[d]^-{B} \ar[r] &0\\
0 \ar[r]& \Z^{n_1} \ar[r]_-{I} 
& \Z^{n_1}\oplus\Z^{n_3} \ar[r]_-{P} 
& \Z^{n_3}  \ar[r] &0
}
\]
commutes 
where $I,I'$ are the obvious inclusions $x \mapsto (x,0)$ 
and $P,P'$ are the obvious projections $(x,y) \mapsto y$. 

Moreover for each $Y \in M_{n_3', n_1} (\Z)$, 
the sequence 
\begin{equation} \label{snake-maps}
\xymatrix{ 
\coker A \ar[r]_-{I} 
& \coker {\left( \begin{smallmatrix} A & Y \\ 0 & B \end{smallmatrix} \right)} \ar[r]_-{P} & \coker B\phantom{.} \ar[r] & 0 \\
\ker B \ar[u]^{[Y]} &  \ker  {\left( \begin{smallmatrix} A & Y \\ 0 & B \end{smallmatrix} \right)} \ar[l]_-{P'} & \ker A \ar[l]_-{I'} & 0 \ar[l]
}
\end{equation}
is exact where we use the same notation $I,P,I',P'$ 
to denote the induced map on cokernels or the restricted maps on kernels, 
and where $[Y]$ is the composition of 
the restriction of $Y$ to $\ker B$ and 
the natural surjection $\Z^{n_1} \to \coker A$. 
\end{proposition}

\begin{proof}
The former assertion is easy to see, 
and the latter follows from the Snake Lemma (see \cite{sm:h}). 
\end{proof}

\begin{lemma} \label{extension-lem}
Let $n,n' \in \{0,1,2,\ldots,\infty\}$ and $A \in M_{n,n'}(\Z)$. 
Let $G$ be an abelian group. 
Then any homomorphism $\eta\colon \ker A \to G$ 
extends to $\zeta \colon \Z^{n'} \to G$ 
such that $\zeta|_{\ker A}=\eta$. 
\end{lemma}

\begin{proof}
Since $\im A$ is a subgroup of the free group $\Z^n$, 
$\im A$ is free. 
Therefore, there exists a homomorphism 
$S \colon \im A \to \Z^{n'}$ such that
\begin{equation} \label{S-def}
A \circ S (x) =  x \quad \text{ for all $x \in \im A$.}
\end{equation}
Let $I\colon \Z^{n'} \to \Z^{n'}$ denote the identity map on $\Z^{n'}$. 
Since $I - SA$ takes values in $\ker A$, 
we can define $\zeta \colon \Z^{n'} \to G$ 
by $\zeta = \eta \circ (I - SA)$. 
It is easy to verify $\zeta|_{\ker A}=\eta$. 
\end{proof}

\begin{proposition} \label{piece-together-prop}
Let $\mathcal{E}$  denote the following exact sequence of abelian groups
\begin{equation} \label{given-exact-seq}
\xymatrix{G_1 \ar[r]^-{\epsilon} & G_2 \ar[r]^-{\gamma} & G_3 \ar[d]^-{0} \\
F_3 \ar[u]^-{\delta} & F_2 \ar[l]_-{\gamma'} & F_1 \ar[l]_-{\epsilon'} 
}
\end{equation}
with $F_1$, $F_2$, and $F_3$ free.  
Suppose that there exist column-finite matrices 
$A \in M_{n_1, n_1'} (\Z)$ 
and $B \in M_{n_3, n_3'} (\Z)$
for some $n_1, n_1', n_3, n_3' \in \{0,1,2,\ldots,\infty\}$ 
with isomorphisms 
\begin{align*}
\alpha_1 &\colon \coker A \to G_1, &
\beta_1 &\colon \ker A \to F_1, \\ \nonumber
\alpha_3 &\colon \coker B \to G_3, &
\beta_3 &\colon \ker B \to F_3. 
\end{align*}

Then there exist a column-finite matrix $Y \in M_{n_1, n_3'} (\Z)$
and isomorphisms 
\begin{align*} 
\alpha_2 \colon \coker \left(\begin{smallmatrix} A & Y \\ 0 & B \end{smallmatrix} \right) \to G_2, &
& \beta_2 &\colon \ker \left(\begin{smallmatrix} A & Y \\ 0 & B \end{smallmatrix} \right) \to F_2
\end{align*}
such that $\alpha_i$ and $\beta_i$ for $i=1,2,3$ 
give an isomorphism (see \eqref{isodef})  from the exact sequence
\begin{equation} \label{snake-maps-again}
\xymatrix{ 
\coker A \ar[r]_-{I} 
& \coker {\left( \begin{smallmatrix} A & Y \\ 0 & B \end{smallmatrix} \right)} \ar[r]_-{P} & \coker B\phantom{.} \ar[d]^0 \\
\ker B \ar[u]^{[Y]} &  \ker  {\left( \begin{smallmatrix} A & Y \\ 0 & B \end{smallmatrix} \right)} \ar[l]_-{P'} & \ker A. \ar[l]_-{I'}
}
\end{equation}
to $\mathcal{E}$, where $I,I'$ and $P,P'$ 
are induced by the obvious inclusions or projections. 
\end{proposition}

\begin{proof}
By a simple calculation or applying the Snake Lemma, 
we see that the sequence \eqref{snake-maps} is exact 
(see Remark~\ref{Snake-remark}). 

We define $\pi_1 \colon \Z^{n_1} \to G_1$ 
(resp.\ $\pi_3 \colon \Z^{n_3} \to G_3$) 
to be the composition of the natural surjection to 
the cokernel and the isomorphism $\alpha_1$ 
(resp.\ $\alpha_3$). 
We first construct a homomorphism 
$\pi_2 \colon \Z^{n_1} \oplus \Z^{n_3} \to G_2$ 
with 
%\begin{align}\label{cond-for-pi2}
%\epsilon \circ \pi_1 &= \pi_2 \circ I, & 
%\gamma \circ \pi_2 &= \pi_3 \circ P 
%\end{align}
\begin{equation}\label{cond-for-pi2}
\xymatrix{
0 \ar[r] & \Z^{n_1} \ar[r]_-{I} \ar[d]^-{\pi_1} 
& \Z^{n_1} \oplus \Z^{n_3} \ar[r]_-{P} \ar@{.>}[d]^-{\pi_2} 
& \Z^{n_3}  \ar[d]^-{\pi_3} \ar[r] & 0\\
\cdots \ar[r]^-{\delta} 
& G_1 \ar[r]^-{\epsilon} & G_2 \ar[r]^-{\gamma} & G_3 \ar[r] & 0
}
\end{equation}
commuting, where
\begin{align*}
I&\colon \Z^{n_1} \ni x \mapsto (x,0) \in \Z^{n_1} \oplus \Z^{n_3}\\ 
P&\colon \Z^{n_1} \oplus \Z^{n_3} \ni (x,y) \mapsto y \in \Z^{n_3}
\end{align*}
are the obvious inclusions, and projections. 

Since $\Z^{n_3}$ is a free abelian group and $\gamma\colon G_2 \to G_3$ 
is surjective, 
there exists a homomorphism $\mu \colon \Z^{n_3} \to G_2$ such that 
\begin{align*} 
\gamma \circ \mu = \pi_3 \qquad\qquad 
\raisebox{0.7cm}{\xymatrix{ & \Z^{n_3} \ar@{.>}[dl]_-{\mu} \ar[d]^-{\pi_3} \\
G_2 \ar@{>>}[r]_-{\gamma} & G_3.
} }
\end{align*}
We define 
$\pi_2 \colon \Z^{n_1} \oplus \Z^{n_3} \to G_2$ 
by 
$$\pi_2 (x,y) := \epsilon(\pi_1(x)) + \mu(y).$$
for $x \in \Z^{n_1}$ and $y \in \Z^{n_3}$. 
Commutativity in \eqref{cond-for-pi2} can be easily verified as 
\begin{align*}
\pi_2(I(x))&=\pi_2(x,0)=\epsilon(\pi_1(x))\\
\gamma (\pi_2(x,y)) &= \gamma \big(\epsilon(\pi_1(x)) + \mu(y)\big)
= 0 + \pi_3(y) = \pi_3 (P (x,y)). 
\end{align*}

Next we construct a homomorphism $Y \colon \Z^{n_3'} \to \Z^{n_1}$ 
such that 
\begin{align} \label{cond-for-Y1}
\pi_2 \circ \left(\begin{smallmatrix} Y \\ B \end{smallmatrix} \right) 
=0 \qquad\qquad 
\xymatrix{ \Z^{n_3'} \ar@{.>}[r]^-{\left({\begin{smallmatrix} Y \\ B \end{smallmatrix}} \right)} 
& \Z^{n_1} \oplus \Z^{n_3} \ar[r]^-{\pi_2} & G_2
}
\end{align}
and 
\begin{align} \label{cond-for-Y2}
\pi_1 \circ Y|_{\ker B} = \delta \circ \beta_3 \qquad\qquad 
\raisebox{0.7cm}{
\xymatrix{ \ker B \ar@{.>}[r]^{Y|_{\ker B}} \ar[d]_{\beta_3} 
& \Z^{n_1}\phantom{.} \ar[d]_{\pi_1} \\
F_3  \ar[r]^\delta & G_1.
} }
\end{align}
By Lemma~\ref{extension-lem}, 
there exists an extension 
$\beta_3' \colon \Z^{n_3'} \to F_3$
of the isomorphism $\beta_3 \colon \ker B \to F_3$. 
Since $\pi_1\colon \Z^{n_1} \to G_1$ is surjective 
and $\Z^{n_3'}$ is a free abelian group, 
there exists a homomorphism $Y_1 \colon \Z^{n_3'} \to \Z^{n_1}$ such that
\begin{equation} \label{Y1-def}
\pi_1 \circ Y_1 = \delta \circ \beta_3' \qquad\qquad 
\raisebox{0.7cm}{
\xymatrix{ 
\Z^{n_3'} \ar@{.>}[r]^-{Y_1} \ar[d]_-{\beta_3'} 
& \Z^{n_1}\phantom{.} \ar[d]_-{\pi_1} \\
F_3  \ar[r]^-\delta & G_1.
} }
\end{equation}
Since $\mu |_{\im B}$ takes values in $\ker \gamma = \im \epsilon$, 
we see that $\mu|_{\im B}\colon \im B \to \im \epsilon$.  
Also, $\im B$ is a subgroup of a free abelian group 
and thus $\im B$ is free abelian.  
Because $\epsilon \circ \pi_1 : \Z^{n_1} \to \im \epsilon$ is surjective, 
there exists a homomorphism $Y_2\colon \im B \to \Z^{n_1}$ such that
\begin{equation*} 
\epsilon \circ \pi_1 \circ Y_2 = \mu|_{\im B} \qquad\qquad 
\raisebox{0.7cm}{
\xymatrix{ 
\im B \ar@{.>}[d]^{Y_2} \ar[r]_{\mu|_{\im B}} 
& G_2\phantom{.} \\
\Z^{n_1}  \ar[r]^{\pi_1} & G_1. \ar[u]_{\epsilon}
} }
\end{equation*}
Define a homomorphism $Y \colon \Z^{n_3'} \to \Z^{n_1}$ by
\begin{equation*} 
Y := Y_1 - Y_2 \circ B,
\end{equation*}
and, in line with our convention, 
we will use the same symbol for the matrix $Y \in M_{n_1, n_3'} (\Z)$ 
that implements this homomorphism.
For $y' \in \Z^{n_3'}$ we have 
\begin{align*}
\pi_2(Y(y'),B(y')) &= \epsilon(\pi_1(Y(y'))) + \mu(B(y')) \\
&= \epsilon \big(\pi_1\big( Y_1(y') - Y_2(B(y') )\big) \big) + \mu(B(y')) \\
&= 0 - \mu(B(y')) + \mu(B(y')) \\
&= 0. 
\end{align*}
This shows \eqref{cond-for-Y1}. 
The equality \eqref{cond-for-Y2} follows from \eqref{Y1-def} 
because 
\[
Y|_{\ker B} = Y_1|_{\ker B} - (Y_2 \circ B)|_{\ker B}
=Y_1|_{\ker B},
\]
and hence
we have 
\begin{align*}
\pi_2 \circ \left(\begin{smallmatrix} A  \\ 0 \end{smallmatrix} \right) 
= \pi_2 \circ I \circ A = \epsilon \circ \pi_1 \circ A = 0. 
\end{align*}
This and \eqref{cond-for-Y1} show 
\[
\pi_2 \circ \left(\begin{smallmatrix} A & Y \\ 0 & B \end{smallmatrix} \right) 
=0. 
\]
Hence the map $\pi_2 \colon \Z^{n_1} \oplus \Z^{n_3} \to G_2$ 
factors through a map 
\[
\coker \left(\begin{smallmatrix} A & Y \\ 0 & B \end{smallmatrix} \right) 
\to G_2
\]
which is denoted by $\alpha_2$. 

Now we shall construct 
a homomorphism 
\[
\beta_2 \colon 
\ker \left(\begin{smallmatrix} A & Y \\ 0 & B \end{smallmatrix} \right) 
\to F_2
\] 
fitting into
\begin{equation}\label{cond-for-beta2}
\xymatrix{
0 \ar[r] & \ker A \ar[d]_-{\beta_1} \ar[r]_-{I'}
& \ker \left({\begin{smallmatrix} A & Y \\ 0 & B \end{smallmatrix}} \right) 
 \ar[r]_-{P'} \ar@{.>}[d]_-{\beta_2} 
& \ker B \ar[d]_-{\beta_3} & \\
0 \ar[r] & F_1\ar[r]^-{\epsilon'} & F_2 \ar[r]^-{\gamma'} & 
F_3 \ar[r]^-{\delta} & G_1
}
\end{equation}
where $I'$ and $P'$ are the restrictions of 
obvious inclusion, and projections, as above.

For $(x',y') \in \ker \left(\begin{smallmatrix} A & Y \\ 0 & B \end{smallmatrix} \right)$, 
 we have 
\begin{align*}
\delta\big( \beta_3 (P'(x',y'))\big) 
& = \pi_1 (Y (y'))
= \pi_1 ( - A (x'))=0. 
\end{align*}
Hence the image of $\beta_3 \circ P'$ is contained 
in $\ker \delta = \im \gamma'$. 
The abelian group 
$\ker \left({\begin{smallmatrix} A & Y \\ 0 & B \end{smallmatrix}} \right)$ 
is free because it is a subgroup of 
the free group $\Z^{n_1'} \oplus \Z^{n_3'}$.  
Therefore there exists a homomorphism 
$\nu \colon \ker \left({\begin{smallmatrix} A & Y \\ 0 & B \end{smallmatrix}} \right) \to F_2$ such that
\begin{align*} 
\gamma' \circ \nu = \beta_3 \circ P' \qquad\qquad 
\raisebox{0.7cm}{ 
\xymatrix{ 
\ker \left({\begin{smallmatrix} A & Y \\ 0 & B \end{smallmatrix}} \right) 
\ar@{.>}[d]_-{\nu} \ar[r]_-{P'} 
& \ker B \ar[d]^-{\beta_3} \\
F_2 \ar[r]_-{\gamma'} & F_3.
} }
\end{align*}
Since 
\[
\gamma' \circ \nu \circ I' = \beta_3 \circ P' \circ I' =0, 
\]
the image of $\nu \circ I'$ is contained in 
$\ker \gamma' = \im \epsilon'$. 
Since $\epsilon'$ is injective, 
there exists a homomorphism $\eta \colon \ker A \to F_1$ 
such that $\epsilon' \circ \eta = \nu \circ I'$. 
By Lemma~\ref{extension-lem}, 
there exists an extension 
$\zeta \colon \Z^{n_1'} \to F_1$
of the homomorphism $\beta_1-\eta \colon \ker A \to F_1$. 
We define 
\[
\beta_2 \colon 
\ker \left(\begin{smallmatrix} A & Y \\ 0 & B \end{smallmatrix} \right) 
\ni (x',y') \mapsto \nu(x',y') + 
\epsilon'(\zeta(x')) \in F_2
\] 
Commutativity `n \eqref{cond-for-beta2} can be verified as 
\begin{align*}
\beta_2(I'(x))
&=\nu(x,0) + \epsilon'(\zeta(x))
=\nu (I' (x)) + \epsilon'(\beta_1(x)-\eta(x))
=\epsilon'(\beta_1(x)) \\
\\
\gamma' (\beta_2(x',y')) 
&= \gamma' ( \nu(x',y') + \epsilon'(\zeta(x'))
= \beta_3(y') + 0 = \beta_3 (P' (x',y')). 
\end{align*}
for $x \in \ker A$ and 
$(x',y') \in \ker \left(\begin{smallmatrix} A & Y \\ 0 & B \end{smallmatrix} \right)$. 

Thus we constructed $Y \in M_{n_1, n_3'} (\Z)$ 
and two homomorphisms $\alpha_2$ and $\beta_2$. 
The diagram 
\begin{align*}
\xymatrix{ \coker A \ar[rr]^-{\iota_*} \ar[rd]_-{\alpha_1} & & 
\coker \left({\begin{smallmatrix} A & Y \\ 0 & B \end{smallmatrix}} \right) \ar[rr]^-{\pi_*} \ar[d]_-{\alpha_2} & & 
\coker B \ar[ddd]^-{0} \ar[dl]^-{\alpha_3} \\ 
& {G_1}\ar[r]^-{\epsilon} & {G_2}\ar[r]^-{\gamma} 
& {G_3}\ar[d]^-{0} & \\ 
& {F_3}\ar[u]^-{\delta} & {F_2} \ar[l]_-{\gamma'} 
& {F_1}\ar[l]_-{\epsilon'} & \\ 
\ker B \ar[uuu]^-{[Y]} \ar[ru]^-{\beta_3} & & 
\ker \left({\begin{smallmatrix} A & Y \\ 0 & B \end{smallmatrix}} \right) 
\ar[ll]_-{\pi_*} \ar[u]^-{\beta_2} & &
\ker A \ar[ll]_-{\iota_*} \ar[ul]_-{\beta_1}}
\end{align*}
commutes 
by \eqref{cond-for-pi2} for the top two squares, 
by \eqref{cond-for-Y2} for the left square, 
by \eqref{cond-for-beta2} for the bottom two squares, 
and trivially for the right square. 
Finally, 
the Five Lemma shows that 
$\alpha_2$ and $\beta_2$ are isomorphisms.
\end{proof}

\begin{remark}\label{Snake-remark}
The content of Proposition~\ref{piece-together-prop} 
can be summarized in the following commutative diagram 
with exact rows and columns: 

\begin{equation} \label{phi-diagram-to-verify}
\xymatrix{
& 0\ar[d] & 0\ar@{.>}[d] & 0\ar[d] &\\
0 \ar[r] & F_1 \ar[r]^-{\epsilon'} \ar[d]_-{\beta_1^{-1}} 
& F_2 \ar[r]^-{\gamma'} \ar@{.>}[d]_-{\beta_2^{-1}} 
& F_3 \ar[d]_{\beta_3^{-1}} \ar @{-}@/^2pc/[d]^-{\delta} |(.77){\cdot} |(.8){\cdot} |(.83){\cdot} &\\
0 \ar[r]& \Z^{n_1'} \ar[r]^-{I'} \ar[d]^-{A} 
& \Z^{n_1'}\oplus\Z^{n_3'} \ar[r]^-{P'} \ar@{.>}[d]^-{\left({\begin{smallmatrix} A & Y \\ 0 & B \end{smallmatrix}} \right)} 
&\Z^{n_3'} \ar[d]^-{B} \ar[r] \ar@{.>}[lld]_(.7){Y} &0\\
0 \ar[r]& \Z^{n_1} \ar@{->}@/_2pc/[d]_-{\delta} |(.17){\cdot} |(.2){\cdot} |(.23){\cdot}\ar[r]_-{I} \ar[d]^-{\pi_1} 
& \Z^{n_1}\oplus\Z^{n_3} \ar[r]_-{P} \ar@{.>}[d]^-{\pi_2} 
& \Z^{n_3}  \ar[d]^-{\pi_3} \ar[r] &0\\
& G_1\ar[r]_-{\epsilon} \ar[d] & G_2 \ar[r]_-{\gamma} \ar@{.>}[d] & G_3 \ar[r] \ar[d] & 0\\
&0&0&0&
}
\end{equation}
The solid lines and $\delta \colon F_3 \to G_1$ are 
the given data, 
and the dotted lines are the one we need to construct. 
We also need to show that $\delta$ factors through 
$Y \colon \Z^{n_3'} \to \Z^{n_1}$, 
the middle column is exact and the six squares commute. 
Then the snake lemma shows that the sequence \eqref{snake-maps}
is isomorphic to the given one $\mathcal{E}$. 
\end{remark}

In order to use the matrix $Y$ found above to define a graph, we will need it to be non-negative. The two ensuing lemmas are the key to arranging this.

\begin{lemma}\label{aver}
Let $n, n' \in \{0,1,2,\ldots,\infty\}$. 
For $A \in M_{n, n'} (\Z)$, 
the following three conditions are equivalent:
\begin{enumerate}
\item For every $i \in \{1,2,\ldots, n\}$ 
there exists $\xi_i \in \Z^{n'}$ 
such that $A\xi_i - \delta_{i} \in (\Z^+)^{n}$. 
\item There exists $A' \in M_{n', n} (\Z)$ 
such that $AA'-I \in M_{n, n} (\Z^+)$. 
\item For every $m \in \{0,1,2,\ldots,\infty\}$ 
and every $Y \in M_{n, m} (\Z)$, 
there exists $Q \in M_{n', m} (\Z)$ such that 
$AQ + Y \in M_{n, m} (\Z^+)$. 
\end{enumerate}
\end{lemma}

\begin{proof}
(1)$\Rightarrow$(2): 
By (1) for each $i \in \{1,2,\ldots, n\}$ 
there exists $\xi_i \in \Z^{n'}$ satisfying 
$A\xi_i - \delta_{i} \in (\Z^+)^{n}$. 
Then $A' := (\xi_1 \xi_2 \cdots \xi_n) \in M_{n', n} (\Z)$ defined by 
$A'(\delta_i) =\xi_i$ for all $i$ satisfies $AA'-I \in M_{n, n} (\Z^+)$. 

(2)$\Rightarrow$(3): 
Take $m$ and $Y \in M_{n, m} (\Z)$. 
We set $|Y| \in M_{n, m} (\Z)$ to be 
the entry-wise absolute value of $Y$; i.e., $|Y| (i,j) := |Y(i,j)|$.  
Let $A'\in M_{n', n} (\Z)$ be as in (2), 
and set $Q := A'|Y| \in M_{n', m} (\Z)$. 
Then we have 
\[
AQ + Y 
= (AA' - I)|Y| + (|Y| + Y) \in M_{n, m} (\Z^+). 
\]

(3)$\Rightarrow$(1): 
Take arbitrary $i \in \{1,2,\ldots, n\}$, 
and apply (3) to $Y \in M_{n, 1} (\Z)$ 
given by $Y(\delta_1)=-\delta_i$ 
to get $Q \in M_{n', 1} (\Z)$ with $AQ + Y \in M_{n, 1} (\Z^+)$. 
If we set $\xi := Q(\delta_1) \in \Z^{n'}$ then 
we have 
\[
A\xi - \delta_{i} = (AQ + Y) (\delta_1) \in (\Z^+)^{n}. 
\qedhere
\]
\end{proof}

The following lemma is analogous to the previous one, and 
is as easy to prove when $n'$ is finite, 
but substantially more complicated when $n'$ is infinite. 
For each $i \in \{1,2,\ldots, n'\}$, 
$\delta_{i}^t \in M_{1, n'} (\Z)$ denotes 
the transpose of $\delta_{i}$, that is, 
the $i$th row of the identity $I \in M_{n', n'} (\Z)$. 

\begin{lemma}\label{bver}
Let $n, n' \in \{0,1,2,\ldots,\infty\}$. 
For $B \in M_{n, n'} (\Z)$, 
the following three conditions are equivalent:
\begin{enumerate}
\item For every $i \in \{1,2,\ldots, n'\}$ 
there exists $\eta_i \in\Z^n$ 
such that $\eta_i^t B - \delta_{i}^t \in M_{1, n'} (\Z^+)$, 
chosen such that for all finite subsets $F \subset \{1,2,\ldots, n\}$, 
there exists a finite subset $G \subset \{1,2,\ldots, n'\}$ 
such that for every $i \not\in G$ and every $j \in F$,
we get $\eta_{i,j}=0$ .
\item There exists $B' \in M_{n', n} (\Z)$ 
such that $B'B-I \in M_{n', n'} (\Z^+)$. 
\item For every $m \in \{0,1,2,\ldots,\infty\}$ 
and every $Y \in M_{m, n'} (\Z)$, 
there exists $Q \in M_{m, n} (\Z)$ such that 
$QB + Y \in M_{m, n'} (\Z^+)$. 
\end{enumerate}
\end{lemma}

\begin{proof}
(1)$\Rightarrow$(2): 
When $n'$ is finite, $B' \in M_{n', n} (\Z)$ 
can be defined so that for each $i \in \{1,2,\ldots, n'\}$ 
the $i$th row of $B$ is $\eta_i^t \in M_{1, n} (\Z)$ 
satisfying $\eta_i^t B - \delta_{i}^t \in M_{1, n'} (\Z^+)$ 
which exist by condition (1). 
When $n'$ is infinite, 
the condition (1) implies that $n$ is also infinite. 
For each integer $k$, 
let $G_k \subset \{1,2,\ldots, n'\}$ 
be a finite set as in (1) for the finite set $F=\{1,2,\ldots,k\}$. 
We define a finite set $G'_k \subset \{1,2,\ldots, n'\}$ 
for $k=1,2,\ldots$ inductively by 
\[
G'_1 := \{1\} \cup G_1,\quad 
G'_{k} := \{k\} \cup G_k \setminus \Big(\bigcup_{j=1}^{k-1} G'_j\Big). 
\]
From this definition, 
it is easy to see that the $\{G'_k\}_{k=1}^\infty$ are mutually disjoint, 
that their union is the whole of $\{1,2,\ldots, n'\} = \N$, 
and $G'_{k} \cap G_{k-1} = \emptyset$ 
for all $k > 1$. 
For $i \in G'_1$, 
choose $\eta_i \in \Z^n$ 
such that $\eta_i^t B - \delta_{i}^t \in M_{1, n'} (\Z^+)$ 
by the first condition of (1). 
For each $i \in G'_{k}$ for $k >1$, 
choose $\eta_i^t \in M_{1, n} (\Z)$ 
such that $\eta_i^t B - \delta_{i}^t \in M_{1, n'} (\Z^+)$ 
and $(\eta_i^t)_{1,j}=0$ for $j=1,2, \ldots, k-1$ 
by the latter condition of (1). 
We set $B' \in M_{n', n} (\Z)$ by $B'_{i,j} = (\eta_i)_{1,j}$. 
We need to check that $B'$ is column-finite, 
which follows from the fact that for each $j \in \{1,2,\ldots, n\}$, 
$B'_{i,j} \neq 0$ implies $i$ is in the finite set 
$\bigcup_{k=1}^{j-1} G'_k$.  
Now it is easy to see $B'B-I \in M_{n', n'} (\Z^+)$. 

(2)$\Rightarrow$(1): 
Take $B'\in M_{n', n} (\Z)$ as in (2), 
For each $i$, let $\eta_i^t \in\Z^n$ 
be the $i$th row of $B$. 
Then we have $\eta_i^t B - \delta_{i}^t \in M_{1, n'} (\Z^+)$. 
Thus we get the former condition of (1). 
This choice of $\eta_i^t$'s 
also satisfies the latter condition of (1) 
if for a given finite subset $F \subset \{1,2,\ldots, n\}$, 
we choose $G \subset \{1,2,\ldots, n'\}$ by 
\[
G = \bigcup_{j \in F} \{ i : B_{i,j} \neq 0\}
\]
which is finite because $B$ is column-finite. 

(2)$\Rightarrow$(3): 
Take $m$ and $Y \in M_{m, n'} (\Z)$. 
We set $|Y| \in M_{m, n'} (\Z)$ to be 
the entry-wise absolute value of $Y$; i.e., $|Y| (i,j) := |Y(i,j)|$.  
Let $B'\in M_{n', n} (\Z)$ be as in (2), 
and set $Q := |Y|B' \in M_{m, n} (\Z)$. 
Then we have 
\[
QB + Y 
= |Y|(B'B - I) + (|Y| + Y) \in M_{m, n'} (\Z^+). 
\]

(3)$\Rightarrow$(2): 
Apply (3) to $m=n$ and $Y= -I$.
\end{proof}

\begin{proposition} \label{positive-Y-prop}
In the situation of Proposition~\ref{piece-together-prop}, assume that $Z\in M_{n_1,n_3'}(\Z)$ is given.
If $A$ satisfies the equivalent conditions of Lemma \ref{aver} or $B$ satisfies the equivalent conditions of Lemma \ref{bver}, then the matrix $Y\in M_{n_1,n_3'}(\Z)$ along with $\alpha_2$, $\beta_2$ inducing the isomorphism may be chosen with the additional property $Y\geq Z$.
\end{proposition}
\begin{proof}
Let $Y'\in M_{n_1,n_3'}(\Z)$ denote a matrix already chosen in Proposition~\ref{piece-together-prop}, along with maps $\alpha_2'$ and $\beta_2'$. Assume first that $A$ satisfies the conditions of Lemma \ref{aver}. Then by (3) of the lemma  we may choose $Q\in M_{n_1',n_3'}(\Z)$ such that
\[
AQ+[Y'-Z]\in M_{n_1,n_3'}(\Z^+)
\]
One checks directly that with
\begin{eqnarray*}
Y&=&AQ+Y'\\
\beta_2&=&\beta_2'\circ \begin{pmatrix}
I&Q\\0&I
 \end{pmatrix}
\\
\alpha_2&=&\alpha_2'
\end{eqnarray*}
the conditions are all met.

When $B$ satisfies the conditions of Lemma \ref{bver}, we choose $Q\in M_{n_1,n_3}(\Z)$ such that
\[
QB+[Y'-Z]\in M_{n_1,n_3'}(\Z^+)
\]
and set
\begin{eqnarray*}
Y&=&Y'+QB\\
\beta_2&=&\beta_2'\\
\alpha_2&=&
\alpha_2'\circ \begin{pmatrix}
I&-Q\\0&I
 \end{pmatrix}
\end{eqnarray*}
\end{proof}

\begin{proposition} \label{unit-Y-prop}
In the situation of Proposition~\ref{piece-together-prop}, assume that $n_1,n_3<\infty$, and that $g_2\in G_2$ is given with  $\alpha_3([\mathbf 1])=\gamma(g_2)$. If $B$ satisfies the condition that for some $1\leq i,j<n_3$ we have
\[
B_{ik}<B_{jk}\qquad 1\leq k<n_3',
\]
then the matrix $Y\in M_{n_1,n_3'}(\Z)$ along with $\alpha_2$, $\beta_2$ inducing the isomorphism may be chosen with the additional property $\alpha_2([\mathbf 1])=g_2$.
\end{proposition}
\begin{proof}
Take $Y'$, $\beta_2'$ and $\alpha_2'$ as in Proposition \ref{piece-together-prop}
and set
\[
g_2'=\alpha_2'([\mathbf 1])-g_2.
\]
Observe that $\gamma(g_2')=0$ because 
\[
\alpha_3([\mathbf 1])=\alpha_3\left(\left(\begin{smallmatrix}\mathbf 1\\\mathbf 1
  \end{smallmatrix}\right)
  \right)=\gamma(g_2)
\]
Hence there exists $\xi\in\Z^{n_1}$ such that $\epsilon(\alpha_1([\xi]))=g_2'$. Choose $Q'\in M_{n_1,n_3}(\Z)$ such that $\xi=Q'\mathbf 1$, which is possible because $n_3\geq 2$.
Find an integer  $c>0$ so that with $Q''\in M_{n_1,n_3}(\Z)$ defined by
\[
(Q'')_{k,\ell}=\begin{cases} 1&\ell=i\\-1&\ell=j\\0&\text{else}
\end{cases}
\]
we have
\[
(Q'+cQ'')B\geq Z-Y '
\]
This is possible because each row of $Q''B$ is identically
\[
\begin{pmatrix}
B_{i,1}-B_{j,1}&B_{i,2}-B_{j,2}&\cdots& B_{i,n_3}-B_{j,n_3}
\end{pmatrix}
\]
which is strictly positive by assumption on $B$. Set $Q=Q'+NQ''$ and $Y=Y'+QB$, and let
\begin{gather*}
\beta_2=\beta_2'\qquad\alpha_2=\alpha_2'\circ \begin{pmatrix}I&-Q\\0&I
\end{pmatrix}.
\end{gather*}
Then obviously $Y\geq Z$, and we get
\begin{eqnarray*}
\alpha_2([\mathbf 1])&=&\alpha_2'\left(\begin{pmatrix}I&-Q\\0&I
  \end{pmatrix}\begin{pmatrix}\mathbf 1\\\mathbf 1
  \end{pmatrix}\right)\\
&=&\alpha_2'\left(\begin{pmatrix}\mathbf 1\\\mathbf 1
 \end{pmatrix}-\begin{pmatrix}Q\mathbf 1\\\mathbf 0
 \end{pmatrix}
\right)
\\
&=&\alpha'_2(\mathbf 1)-\alpha_2'(I(Q\mathbf 1))\\
&=&\alpha'_2(\mathbf 1)-\epsilon\circ\alpha_1([\xi])\\
&=&\alpha_2'(\mathbf 1)-g_2'\\
&=&g_2
\end{eqnarray*}
\end{proof}

\begin{proposition} \label{split-Y-prop}
In the situation of Proposition~\ref{piece-together-prop}, assume that   $n_1,n_3<\infty$, that $F_3=0$, and that there exists a splitting map $\sigma:G_3\to G_2$ for $\gamma$.
Let $g_2\in G_2$ be given, set $g_3=\gamma(g_2)\in G_3$ and let $g_1$ be the unique element of $G_1$ with $\epsilon(g_1)=g_2-\sigma(g_3)$. If
\[
\alpha_1(\mathbf 1)=g_1\qquad \alpha_3(\mathbf 1)=g_3,
\]
then there exist maps 
\begin{align*} 
\alpha_2 \colon \coker \left(\begin{smallmatrix} A & 0 \\ 0 & B \end{smallmatrix} \right) \to G_2, &
& \beta_2 &\colon \ker \left(\begin{smallmatrix} A & 0 \\ 0 & B \end{smallmatrix} \right) \to F_2
\end{align*}
such that with $Y=0$, the collection of maps $\alpha_i$ and $\beta_i$ for $i=1,2,3$ 
provide an isomorphism, and  $\alpha_2([\mathbf 1])=g_2$.
\end{proposition}
\begin{proof}
Let
\[
\alpha_2\left(\left(\begin{smallmatrix}\xi\\\eta
  \end{smallmatrix}\right)\right)=\epsilon(\alpha_1([\xi]))+\sigma(\alpha_3([\eta]))
\]
and
\[
\beta_2\left(\left(\begin{smallmatrix}\xi\\0
  \end{smallmatrix}\right)\right)=
\epsilon'(\beta_1(\xi)).
\]
\end{proof}

\section{Gluing graphs}\label{gluing}

Suppose two graphs $E_1$ and $E_3$ are given along with groups $G_2$ and $F_2$ that satisfy the following  diagram $\mathcal E$ given by
\begin{equation} \label{half-six-term}
\xymatrix{
K_0(C^*(E_1)) \ar[r]^-{\epsilon} & G_2 \ar[r]^-{\gamma} & K_0(C^*(E_3)) \ar[d]^-{0} \\
K_1(C^*(E_3)) \ar[u]^-{\partial_1} & F_2 \ar[l]_-{\gamma'} & K_1(C^*(E_1)). \ar[l]_-{\epsilon'} 
}
\end{equation}
In the present section we investigate circumstances under which it is possible to glue together $E_1$ and $E_3$ to form a third graph $E_2$ whose $C^*$-algebra has $\mathcal E$ as it six-term exact sequence in $K$-theory.  We shall see that the results of the previous section allow us to perform such a gluing under very modest assumptions on either $E_1$ or $E_3$, realizing the sequence of groups in $\mathcal{E}$.  However, since there are natural obstructions on the order on $G_2$ for this ordered group to originate from a graph $C^*$-algebra, we will need to impose further restrictions before being able to realize the ordered sequence $\mathcal E^+$ consisting of an exact sequence of partially ordered groups. The necessity of these conditions follow from the fullness issues considered in \cite{ET1} and \cite{ERR2}, but for the reader's convenience we shall develop them by much more elementary methods at the end of the section.

 \subsection{Adhesive graphs}

\begin{definition}\label{leftadh}
We say that the graph $E$ is \emph{left adhesive} if 
for any $v_0\in E^0$ there exist distinct $e_0,e_1,\dots,e_n\in E^1$ such that  $V_{v_0}=\{r(e_k)\mid k=1,\dots,n\}$ 
\begin{enumerate}[(i)]
\item $r(e_0)=v_0$
\item $s(e_k)\in V_{v_0}$ for all $k=0,1,\dots, n$
\item $V_{v_0}\subseteq E^0_\textnormal{reg}$
\end{enumerate}
where $V_{v_0} :=\{r(e_k)\mid k=1,\dots,n\}$.
\end{definition}

\begin{definition}\label{rightadh}
We say that the graph $E$ is \emph{right adhesive} if
for any $v_0\in (E^0)_\textnormal{reg}$ there exist distinct $e_0,e_1,\dots,e_n\in E^1$ such that
\begin{enumerate}[(i)]
\item $s(e_0)=v_0$
\item $r(e_k)\in W_{v_0}$ for all $k=0,1,\dots, n$
\item $W_{v_0}$ satisfies that for any $w\in E^0$, the set $\{v_0\in E^0_\textnormal{reg}\mid w\in W_{v_0}\}$ is finite
\end{enumerate}
where $W_{v_0}:=\{r(e_k)\mid k=1,\dots,n\}$.
\end{definition}

The reader is requested to note the similarities between these concepts, and how there is only partial symmetry. The relevance of adhesiveness in our situation is explained by the following lemma.

\begin{lemma}\label{adhrelevance}
When $E=(E^0,E^1,r,s)$ is left adhesive, then $R_E-I$ satisfies the equivalent conditions of Lemma \ref{aver}. When $E$ is right adhesive, 
$R_E-I$ satisfies the equivalent conditions of Lemma \ref{bver}.
\end{lemma}
\begin{proof}
In the first case, define $\xi_{v_0}\in \Z^{E^0_\textnormal{reg}}$ by
\[
\xi_{v_0}=\sum_{w\in V_{v_0}}\delta_w.
\]
Then $(R_E-I)\xi_{v_0}\geq \delta_{v_0}$ because 
\[
(R_E-I)\xi_{v_0}=\begin{cases}
|\{e\in E^1\mid r(e)=v, s(e)\in V_{v_0}\}|-1&\text{if }v\in V_{v_0}\\
|\{e\in E^1\mid r(e)=v, s(e)\in V_{v_0}\}|&\text{if }v\not \in V_{v_0}.
\end{cases}
\]
For every $v\in V_{v_0}$ we have at least one edge starting in $V_{v_0}$ and ending in $v$, so $(R_E-I)\xi_{v_0}\geq 0$. We also 
obviously have $(R_E-I)\xi_{v_0}\geq \delta_{v_0}$ in the case when $v\not\in V_{v_0}$. When  
$v_0\in V_{v_0}$ we have that $v_0=r(e_i)$ for $e_i\not=e_0$, so that two different edges start in $V_{v_0}$ and end in $v_0$, as required.

In the second case, define $\eta_{v_0}\in \Z^{E^0}$ by
\[
\eta_{v_0}=\sum_{w\in W_{v_0}}\delta_w
\]
and note again that
\[
\eta_{v_0}^t(R_E-I)=\begin{cases}
|\{e\in E^1\mid s(e)=v, r(e)\in W_{v_0}\}|-1&\text{if }v\in W_{v_0}\cap E_\textnormal{reg}^0\\
|\{e\in E^1\mid s(e)=v, r(e)\in W_{v_0}\}|&\text{if }v\not \in W_{v_0}\cap E_\textnormal{reg}^0
\end{cases}
\]
to get the desired conclusion $\eta_{v_0}^t(R_E-I)\geq\delta_{v_0}^t$. We also see by Condition~(iii) that  for finite $F\subseteq E^0$ we may take
\[
G=\bigcup_{w\in F}\{v_0\in E_\textnormal{reg}^0\mid w\in W_{v_0}\}
\] 
to arrange that $\eta_{v,w}=0$ when $v\in F$ and $w\not\in G$.
\end{proof}

Checking adhesiveness by the definition is rarely necessary, and in each case below we will be able to simply appeal to one of these simpler conditions:

\begin{lemma}\label{adhcheck}
Let $E=(E^0,E^1,r,s)$ be a graph. Then $E$ is left adhesive when any of the conditions hold:
\begin{itemize}
\item[$(\ell 1)$] $E^0=E^0_\textnormal{reg}$, and each $v_0\in E^0$ supports two loops. 
\item[$(\ell 2)$]  For each $v_0\in E^0$, there exist edges $e_1,\dots,e_n\in E^1$ forming a cycle so that each $s(e_k)\in E^0_\textnormal{reg}$, and $e_0\in E^1$ so that $e_0\not=e_1$, $s(e_0)=s(e_1)$, and $r(e_0)=v_0$
\end{itemize}
In addition, $E$ is right adhesive when any of the following conditions hold:
\begin{itemize}
\item[$(r0)$] $E^0_\textnormal{reg}=\emptyset$
\item[$(r1)$] Each $v_0\in E^0_\textnormal{reg}$ supports two loops. 
\item[$(r2)$] $E^0_\textnormal{reg}$ is finite,  and for each $v_0\in E^0_\textnormal{reg}$, there exist edges $e_1,\dots,e_n\in E^1$ forming a cycle, and $e_0\in E^1$ so that $e_0\not=e_1$, $r(e_0)=r(e_1)$, and $s(e_0)=v_0$.
\end{itemize}
\end{lemma}
\begin{proof}
In $(\ell1)$ and $(r1)$, we choose at each $v_0\in E_\textnormal{reg}^0$ the two loops as our $e_0,e_1$. Claim $(r0)$ is obvious, and for $(\ell2)$ and $(r2)$ we choose the indicated sets of edges, noting in the latter case that the finiteness condition is automatically true.
\end{proof}

\begin{proposition} \label{six-term-graph-splice}
Let $\mathcal{E}:$ 
\begin{equation} \label{given-exact-seq-again}
\xymatrix{G_1 \ar[r]^-{\epsilon} & G_2 \ar[r]^-{\gamma} & G_3 \ar[d]^-{0} \\
F_3 \ar[u]^-{\delta} & F_2 \ar[l]_-{\gamma'} & F_1 \ar[l]_-{\epsilon'} 
}
\end{equation}
be an exact sequence of abelian groups 
with $F_1$, $F_2$, and $F_3$. Let $E_1 = (E_1^0, E_1^1, r_{E_1}, s_{E_1})$ and $E_3 = (E_3^0, E_3^1, r_{E_3}, s_{E_3})$ be graphs such that 
\[
\alpha_i:K_0(C^*(E_i))\cong G_i\text{ and }\beta_i:K_1(C^*(E_i))\cong F_i
\]
are given for $i = 1,3$. 

When $E_1$ is left adhesive or $E_3$ is right adhesive, there exists a graph $E_2=(E_2^0, E_2^1, r_{E_2}, s_{E_2})$ with the properties
\begin{enumerate}[(1)]
\item $E_2^0 = E^0_1 \sqcup E_3^0$;
\item $E_2^1$ is equal to the disjoint union of $E_1^1$ and $E_3^1$ together with 
\begin{enumerate}[(a)]
\item a finite nonzero number of edges from each  $v\in(E_3^0)_\textnormal{reg}$ to some vertices in $E_1^0$
\item an infinite number of edges from each  $v\in(E_3^0)_\textnormal{sing}$ to each vertex in $E_1^0$
\end{enumerate}
\end{enumerate}
so that with $\I$ the ideal of $C^*(E_2)$ given by $E_1$, $\I$ is essential, and there exist $\alpha_2:K_0(C^*(E_2))\cong G_2$ and $\beta_2:K_1(C^*(E_2))\cong F_2$ so that $\Ksix(C^*(E_2) ,\I)$ is isomorphic to $\mathcal E$ via the maps $\alpha_i$ and $\beta_i$, $i=1,2,3$.
\end{proposition}

\begin{proof}
Let $R_{E_1}$ and $R_{E_3}$ denote the vertex matrices of $E_1$ and $E_3$, respectively, and set $A=R_{E_1}-I$ and $B=R_{E_3}-I$. By Lemma \ref{adhrelevance}, either $A$ satisfies the conditions of Lemma \ref{aver} or $B$ those of Lemma \ref{bver}, so we may apply Proposition \ref{positive-Y-prop} to find $Y$ such that $Y\geq Z$ where we set 
\[
Z=
\begin{pmatrix}
1 &1&\cdots\\
0 &0&\cdots\\
\vdots&&\\
\end{pmatrix}.
\]
Thus we obtain that $Y$ is nonnegative, with a positive entry in each column. Using $Y$ to read off how many edges to add, and adding an infinite number of edges from every $v\in (E_3)_\textnormal{sing}^0$ to every $w\in E_1^0$ we create $E_3$ with regular vertex matrix $R_{E_3}$ so that $R_{E_3}-I$ takes the form $\left(\begin{smallmatrix}A&Y\\0&B\end{smallmatrix}\right)$. Note that we have arranged that $E_1^0$ is saturated and hereditary in $E_2$, and that 
\[
 (E_2^0)_\textnormal{reg}\cap E_3^0= (E_3^0)_\textnormal{reg}.
\]
Hence $E_1^0$ defines a gauge-invariant ideal $\I$. To show that $\I$ is essential, it suffices to prove that $\I$ nontrivially intersects every nonzero gauge-invariant ideal. Let such  an ideal be given by a hereditary and saturated set $H$ along with a set of breaking vertices $B$.  It then suffices to prove that $H\cap E_1^0\not=\emptyset$. This follows by noting that if $H\subseteq E_ 3^0$, then some $v\in H\cap E_3^0$ may be chosen, and since this vertex has at least one edge to $E_1^0$, we have found the desired contradiction.
\end{proof}

\begin{proposition} \label{six-term-graph-splice-stenotic}
In the situation of Proposition \ref{six-term-graph-splice}, if 
one of the following conditions holds:
\begin{enumerate}[(i)]
\item
$E_1^0$ is the smallest hereditary and saturated subset of itself containing $v_1,\dots,v_n$ for some finite choice of $v_i$
\item
$E_3$ is transitive, and  either
\begin{enumerate}[(1)]
\item $(E^0_3)_\textnormal{sing}\not=\emptyset$; or
\item $|E^0_3|=\infty$
\end{enumerate}
\end{enumerate}
then  $\I$ may be chosen stenotic.
\end{proposition} 
\begin{proof} 
For (i), arrange the vertices of $E_1$ such that $v_1,\dots,v_n$ are listed first. Choosing $Y$ dominating 
\[
Z=
\begin{pmatrix}
1 &1&1&\cdots\\
\vdots&&&\\
1&1&1&\cdots\\
0&0&0&\cdots\\
\vdots&&&&
\end{pmatrix}
\]
in the previous proof we may arrange that there is an edge from each vertex in $E_3$ to each hereditary and saturated subset in $E_1$. Let $\J$ be an ideal of $C^*(E_2)$ which, as above, we may assume is gauge invariant and hence given by $(H,B)$. Since no vertex in $E_3$ is breaking for any subset of $E_1$, we see that if $\J\not\subseteq \I$, $H$ must intersect $E_3^0$. But then, by our construction and the condition that $H$ is hereditary, we see that $E_1^0\subseteq H$, and hence that  $\I \subseteq \J$.

In the case (ii)(2) we choose instead $Y$ dominating 
$Z=I$
and for (ii)(1) recall that every singular vertex of $E_3$ emits, by our construction, an edge to any vertex of $E_1$. Now when $\J\not\subseteq \I$ and $(H,B)$ are given as above, we again get that $H\cap E_3^0\not=\emptyset$, which implies by transitivity that $E_3^0\subseteq H$. Then if $v\in E_3^0$ is singular, since it emits an edge to each vertex in $E_1$, we get that $H=E_2^0$ since $H$ is hereditary. Similarly, if $v_1,v_2,\dots$ are regular in $E_3^0$, we have that $\{v_1,\dots,v_n\}$ emits to the first $n$ vertices of $E_1$ so that again $H=E_2^0$.
\end{proof}

\subsection{Order obstructions}

\begin{proposition}\label{inx-case}
Let $\A$ be a $C^*$-algebra with real rank zero, and let $\I$ be an ideal that is simple and purely infinite. If $\pi:\A\to \A/I$ denotes the quotient map, we have
\[
K_0(\A)^+=\{x\in K_0(\A)\mid \pi_*(x)\geq 0\}
\] 
\end{proposition}
\begin{proof}
  One inclusion is clear since $\pi_*$ is positive, so let $x\in K_0(\A)$ be given with $\pi_*(x)\geq 0$. Since $\A$ has real rank zero, we may choose $x'\in K_0(\A)^+$ so that $\pi_*(x)=\pi_*(x')$, and since $\I$ is simple and purely infinite, we have that $x-x'\geq 0$ in $K_0(\I)$, and hence in $K_0(\A)$. We conclude that $x=(x-x')+x'\geq 0$.  
\end{proof}

\begin{proposition}
Let a $C^*$-algebra $\A$ be given with an approximate unit of projections, and with $\I$ a largest ideal.  When there exists a projection $p\in \A$ such that 
\begin{enumerate}[(i)]
\item $p\not\in \I$
\item $[p]=0$ in $K_0(\A)$
\end{enumerate}
then $K_0(\A)^+=K_0(\A)$.
\end{proposition}
\begin{proof}
We have noted that $K_0(\A)=K_0(\A)^+-K_0(\A)^+$ whenever $\A$ has an approximate unit of projections, so it suffices to prove that $-K_0^+(\A)\subseteq K_0(\A)$.
Note that $p\in \A\subseteq \A\otimes\K$ is full since it is not an element of the largest ideal $\I$. Hence for any projection  $q\in A\otimes \K$ there exists $n$ with the property that $q$ is Murray-von Neumann equivalent to a subprojection $r$ of $p^{\oplus n}$, and we get that
\[
-[q]_0=n[p]_0-[q]_0=[p^{\oplus n}-r]_0\geq 0
\]
in $K_0(\A)$. 
\end{proof}

\begin{proposition}
Let $C^*(E)$ be a graph $C^*$-algebra with a gauge-invariant ideal $\I$ such that $C^*(E)/\I$ is purely infinite and simple. Then there exists a projection $p\in C^*(E)$ such that $p\not\in \I$ and such that $[p]_0=0$ in $K_0(\A)$.
\end{proposition}
\begin{proof}
We may choose a subgraph $E_3\subseteq E$ so that $C^*(E)/I\simeq C^*(E_3)$, and in the subgraph $E_3$ there exists a vertex $v\in E_3^0$ supporting two different cycles $\xi,\eta\in E_3^*$. Let $p=p_v-s_\xi s_\xi^*\in C^*(E)$.  Then $[p]_0=0$. But since $p\geq s_\eta s_\eta^*$ we have that $p\not \in \I$.
\end{proof}

\begin{corollary}\label{graph-ext-order}
Let $E$ be a graph, let $\I$ be an ideal of $C^*(E)$, and let $\pi : C^*(E) \to C^*(E) / \I$ be the quotient map.  Then the following two statements hold.
\begin{enumerate}[(1)]
\item If $C^*(E)$ has real rank zero and $\I$ is purely infinite and simple, then
\[
K_0(C^*(E))^+=\{x\in K_0(C^*(E))\mid \pi_*(x)\geq 0\}.
\] 
\item If $\I$ is the largest ideal of $C^*(E)$ and $C^*(E)/\I$ is purely infinite and simple, then
\[
K_0(C^*(E))^+=K_0(C^*(E)).
\] 
\end{enumerate}
\end{corollary}

%%%%%%%%%%%%%%%%%%%%%%%%%%%%%%%%%%%%%%%%%%%%%%%%%%%%%% 
\section{Six-term exact sequences realized by graph $C^*$-algebras} \label{ranges-sec}
%%%%%%%%%%%%%%%%%%%%%%%%%%%%%%%%%%%%%%%%%%%%%%%%%%%%%%

In this section we consider the range of the invariant $\Ksix (\A,\I)$ for various classes of graph $C^*$-algebras that have $\Ksix (\A,\I)$ as a complete stable isomorphism invariant.

It was proven in \cite[Theorem~4.7]{ET1} that the six-term exact sequence $\Ksix(C^*(E),\I_\textnormal{max})$ is a complete stable isomorphism invariant when $C^*(E)$ is a graph $C^*$-algebra, $\I_\textnormal{max}$ is a largest ideal in $C^*(E)$, and $\I_\textnormal{max}$ is an AF-algebra.  It was also proven in \cite[Corollary~6.4]{ERR} that the six-term exact sequence $\Ksix(C^*(E),\I_\textnormal{min})$ is a complete stable isomorphism invariant when $C^*(E)$ is a graph $C^*$-algebra, $\I_\textnormal{min}$ is a smallest nontrivial ideal in $C^*(E)$, and $C^*(E)/\I_\textnormal{min}$ is an AF-algebra.  In the first case, $C^*(E)/\I_\textnormal{max}$ is simple and hence either purely infinite or AF.  If $C^*(E)/\I_\textnormal{min}$ is AF, then $C^*(E)$ is an AF-algebra and the ordered group $K_0(C^*(E))$ is a complete stable isomorphism invariant.  Thus the case that we are concerned with the six-term exact sequence is when $C^*(E)/\I_\textnormal{max}$ is purely infinite.  Likewise, $\I_\textnormal{min}$ is simple and hence either purely infinite or AF.  When $\I_\textnormal{min}$ is AF, then $C^*(E)$ is AF and $K_0(C^*(E))$ is again a complete stable isomorphism invariant.  Thus the case that we are concerned with is when $\I_\textnormal{min}$ is purely infinite.

\begin{theorem} \label{largest-ideal-range-thm}
Let $C^*(E)$ be a graph $C^*$-algebra with a largest nontrivial ideal $\I$ such that $\I$ is an AF-algebra and $C^*(E)/\I$ is purely infinite.  Then $\Ksix (C^*(E),\I)$ is a complete stable isomorphism invariant within this class, and the range of this invariant is all six-term exact sequences of countable abelian groups
$$\xymatrix{
R_1 \ar[r]^-{\epsilon} & G_2 \ar[r]^-{\gamma} & G_3 \ar[d]^-0 \\
F_3 \ar[u]^-\delta & F_2 \ar[l]_-{\gamma'} & 0 \ar[l]_-{0} 
}
$$
where $F_2$ and $F_3$ are free abelian groups, $R_1$ is a Riesz group, and $G_2$ and $G_3$ have the trivial pre-ordering (i.e., $G_i^+ = G_i$ for $i=2,3$).
\end{theorem}

\begin{proof}
It follows from \cite[Theorem~4.7]{ET1} that $\Ksix (C^*(E),\I)$ is a complete stable isomorphism invariant within this class.  The necessity of the form of the exact sequence stated above follows from the fact that the descending map $\partial_0 : K_0(C^*(E)/\I) \to K_1(\I)$ is always zero \cite[Theorem~4.1]{CET},  by well-known facts about the ordered $K$-theory of 
AF or purely infinite $C^*$-algebras, and from Corollary \ref{graph-ext-order}(2).

To see that all such exact sequences are attained, we know by Proposition~\ref{realize-simple-K-theories-AF-prop} that there exists a row-finite graph with no sinks $E_1$ such that $C^*(E_1)$ is an AF-algebra and $K_0(C^*(E_1))$ is order isomorphic to $R_1$. By Proposition~\ref{realize-simple-K-theories-prop} there exists a row-finite graph with no sinks $E_3$ satisfying Conditions (1)--(3) of Proposition~\ref{realize-simple-K-theories-prop}, with $K_0(C^*(E_3)) \cong G_3$ and $K_1(C^*(E_3)) \cong F_3$.  Then $C^*(E_3)$ is purely infinite and simple.  By Condition~(2) of Proposition~\ref{realize-simple-K-theories-prop} each vertex of the graph $E_3$ contains two loops, and hence $E_3$ is right adhesive appealing to condition $(r1)$ of Lemma \ref{adhcheck}. We have arranged that $|E_3^0|=\infty$ and that $E_3$ is transitive, so by Proposition \ref{six-term-graph-splice-stenotic}(ii)(2) we may glue together $E_1$ and $E_3$ in such a way that the 
ideal $\I$ of $C^*(E_2)$ corresponding to $E_1$ is stenotic. Since $C^*(E_2) / \I \cong C^*(E_3)$ is simple, it follows that  $\I$ is the largest ideal of $C^*(E_2)$.  We have arranged that it is AF. Moreover, Corollary \ref{graph-ext-order}(2) implies $K_0(C^*(E_2))^+ = K_0(C^*(E_2))$, so since we have assumed that $G_2^+ = G_2$, our constructed map $\alpha_2$ will automatically  be an order isomorphism.  Hence $\Ksix (C^*(E_2), \I$ is order isomorphic to the required six-term exact sequence. 
\end{proof}

\begin{theorem} \label{smallest-ideal-range-thm}
Let $C^*(E)$ be a graph $C^*$-algebra with a smallest nontrivial ideal $\I$ such that $C^*(E)/\I$ is an AF-algebra and $\I$ is purely infinite.  Then $\Ksix (C^*(E),\I)$ is a complete stable isomorphism invariant within this class, and the range of this invariant is all six-term exact sequences of countable abelian groups
$$\xymatrix{
G_1 \ar[r]^-{\epsilon} & G_2 \ar[r]^-{\gamma} & R_3 \ar[d]^-0 \\
0 \ar[u]^-0 & F_2 \ar[l]_-{0} & F_1 \ar[l]_-{\epsilon'} 
}
$$
where $F_1$ and $F_2$ are free abelian groups, $R_3$ is a Riesz group, the group $G_1$ has the trivial pre-ordering $G_1^+=G_1$, and $G_2$ has the pre-ordering $G_2^+ = \{ x \in G_2 : \gamma(x) \geq 0 \}$. 
\end{theorem}

\begin{proof} 
It follows from \cite[Corollary~6.4]{ERR} that $\Ksix (C^*(E),\I)$ is a complete stable isomorphism invariant within this class.  The necessity of the form of the exact sequence stated above follows from the fact that the descending map $\partial_0 : K_0(C^*(E) /\I) \to K_1(\I)$ is always zero \cite[Theorem~4.1]{CET}, the fact that the $K_0$-group of an AF-algebra is always a Riesz group, the fact that the $K_1$-group of an AF-algebra is always zero, the fact that the homomorphisms that appear are always order homomorphisms, and from Corollary~\ref{graph-ext-order}(2).

To see that all such exact sequences are attained,  we know by Proposition~\ref{realize-simple-K-theories-AF-prop} that there exists a row-finite graph with no sinks $E_3$ such that $C^*(E_3)$ is an AF-algebra and $K_0(C^*(E_3))$ is order isomorphic to $R_3$. By Proposition~\ref{realize-simple-K-theories-prop} there exists a row-finite graph with no sinks $E_1$ satisfying Conditions (1)--(3) of Proposition~\ref{realize-simple-K-theories-prop}, with $K_0(C^*(E_1)) \cong G_1$ and $K_1(C^*(E_1)) \cong F_1$.  Then $C^*(E_1)$ is purely infinite and simple.  By Condition~(2) of Proposition~\ref{realize-simple-K-theories-prop} each vertex of the graph $E_1$ contains two loops.  Thus the regular vertex matrix $R_{E_1}-I$ has non-negative entries and positive entries down its diagonal.  It follows from Proposition~\ref{six-term-graph-splice} that there exists a graph $E_2$ satisfying Conditions~(1)--(4) of Proposition~\ref{six-term-graph-splice}.  Furthermore, $C^*(E_2) / \I$ is an AF-algebra, and $\I$ is Morita equivalent to $C^*(E_1)$ and thus purely infinite and simple.  Because every vertex of $E^0_3$ has a finite and nonzero number of edges from this vertex to $E^0_1$, and because $E_1$ is strongly connected, it follows that any nonempty hereditary subset of $E_2$ must contain $E_1^0$.   Thus $\I$ is the smallest ideal of $C^*(E_2)$.   Moreover, Corollary \ref{graph-ext-order}(1) implies $K_0(C^*(E_2))^+ =  \{ x \in K_0(C^*(E_2)) : \pi_* (x) \geq 0 \}$ and thus in the commutative diagram \eqref{isodef} we have that $\phi$ is an order isomorphism.  Hence $\Ksix (C^*(E_2), \I)$ is order isomorphic to the required six-term exact sequence. 
\end{proof}

Next we consider graph $C^*$-algebras with a unique nontrivial ideal.  Recall that if $E$ is a graph and $\I$ is a unique proper ideal in $C^*(E)$, then $\I$ and $C^*(E)/\I$ are simple.  Since $\I$ and $C^*(E)/\I$ are both graph $C^*$-algebras, by \cite[Lemma~1.3]{DHS} and \cite[Corollary~3.5]{BHRS}, and since any simple graph $C^*$-algebra is either an AF-algebra or a Kirchberg algebra, there are four cases to consider:

$ $

Type $\inin$ : $\I$  and $C^*(E)/\I$ are both Kirchberg algebras.

Type $\fiin$ : $\I$ is an AF-algebra and $C^*(E)/\I$ is a Kirchberg algebra.

Type $\infi$ : $\I$ is a Kirchberg algebra and $C^*(E)/\I$ is an AF-algebra.

Type $\fifi$ : $\I$ and $C^*(E)/\I$ are both AF-algebras.

$ $

\begin{theorem} \label{one-ideal-range-thm}
If $C^*(E)$ is a graph $C^*$-algebra with a unique nontrivial ideal $\I$, then $\Ksix (C^*(E),\I)$ is a complete stable isomorphism invariant within this class.  The range of this invariant for the four possible types describe the six-term exact sequences 
$$\xymatrix{
G_1 \ar[r]^-{\epsilon} & G_2 \ar[r]^-{\gamma} & G_3 \ar[d]^-0 \\
F_3 \ar[u]^-\delta & F_2 \ar[l]_-{\gamma'} & F_1 \ar[l]_-{\epsilon'} 
}
$$
occurring as follows, where each $F_1$, $F_2$, and $F_3$ are free abelian groups.

\noindent Type $\inin$ : All the groups $G_1$, $G_2$, and $G_3$ have the trivial pre-ordering (i.e., $G_i^+ = G_i$ for $i=1,2,3$).

\noindent Type $\fiin$ : $F_1=0$, $G_1$ is a simple Riesz group,  and $G_2$ and $G_3$ have the trivial pre-ordering (i.e., $G_i^+ = G_i$ for $i=2,3$).

\noindent Type $\infi$ : $F_3=0$, $G_3$ is a simple Riesz group, the group $G_1$ has the trivial pre-ordering $G_1^+=G_1$, and $G_2$ has the pre-ordering $G_2^+ = \epsilon (G_1) \sqcup \{ x \in G_2 : \gamma (x) > 0 \}$.

\noindent Type $\fifi$ : $F_1=0$, $F_2=0$, $F_3=0$, $G_1$, $G_2$, and $G_3$ are Riesz groups, $G_1$ and $G_3$ are simple ordered groups, and the sequence is lexicographically ordered (cf. \cite{handelman}); i.e., $\epsilon(G_1^+) = \epsilon(G_1) \cap G_2^+$ and $\gamma(G_2^+) = G_3^+$.
\end{theorem}

\begin{proof}
It follows from \cite[Theorem~4.5]{ET1} that $\Ksix (C^*(E),\I)$ is a complete stable isomorphism invariant for the class of graph $C^*$-algebras with a unique nontrivial ideal.  Also, the necessity of the forms of the exact sequences stated for the four types follows from the fact that the descending map $\partial_0 : K_0(C^*(E)/\I) \to K_1(\I)$ is always zero \cite[Theorem~4.1]{CET}, the fact that the $K_0$-group of an AF-algebra is always a Riesz group, the fact that the $K_1$-group of an AF-algebra is always zero, the fact that the homomorphisms that appear are always order homomorphisms, and from Corollary \ref{graph-ext-order}.  To see that all sequences in the four types are attained, we consider the four cases separately.

\noindent \textbf{Case $\inin$:}  By Proposition~\ref{realize-simple-K-theories-prop} there exist row-finite graphs with no sinks $E_1$ and $E_3$, each satisfying Conditions (1)--(3) of Proposition~\ref{realize-simple-K-theories-prop}, with $K_0(C^*(E_1)) \cong G_1$ and $K_1(C^*(E_1)) \cong F_1$ and with $K_0(C^*(E_3)) \cong G_3$ and $K_1(C^*(E_3)) \cong F_3$.  Then $C^*(E_1)$ and $C^*(E_3)$ are purely infinite and simple, and in fact both are left and right adhesive.
Thus it follows from Proposition~\ref{six-term-graph-splice} that $E_1$ and $E_3$ may be glued to obtain a graph $E_2$ with an essential ideal $\I$ such that $C^*(E_2) / \I \cong C^*(E_3)$ is purely infinite and simple, and $\I$ is Morita equivalent to $C^*(E_1)$ and thus also purely infinite and simple.  It follows that  $\I$ is the unique proper ideal of $C^*(E_2)$ and that $C^*(E_2)$ is of type $\inin$.  Moreover, the commutative diagram that appears in \eqref{isodef} has that $\phi$ is an order isomorphism due to the fact that $K_0(C^*(E_2))^+ = K_0(C^*(E_2))$ by Corollary \ref{graph-ext-order} and $G_2^+ = G_2$.  Thus $\Ksix (C^*(E_2), \I)$ is order isomorphic to the required six-term exact sequence.

\noindent \textbf{Case $\fiin$:}  This is a special case of  Theorem \ref{largest-ideal-range-thm}.

\noindent \textbf{Case $\infi$:}  This is a special case of  Theorem \ref{smallest-ideal-range-thm}.

\noindent \textbf{Case $\fifi$:} It follows from Proposition~\ref{realize-simple-K-theories-AF-prop} that there exists a row-finite graph with no sinks $E$ such that $C^*(E)$ is an AF-algebra and $K_0(C^*(E))$ is order isomorphic to $R_2$.   Since the extension $0 \to R_1 \to R_2 \to R_3 \to 0$ is order exact and $R_1$ and $R_3$ are simple, it follows that $R_2$ has exactly one nontrivial order ideal, namely $\epsilon(R_1)$.  Thus there exists a unique nontrivial ideal $\I \triangleleft C^*(E)$ with $\Ksix(C^*(E), \I)$ isomorphic to the given sequence.  Furthermore, since ideals and quotients of AF-algebras are AF-algebras, we see that $C^*(E)$ is of type $\fifi$.
\end{proof}

Next we consider the range of the six-term exact sequence for graph $C^*$-algebras that are unital extensions of Kirchberg algebras. We do not as yet have a complete classification theory within this class, but in the $\inin$ case such a result was provided adding only the class of the unit to $\Ksix (C^*(E),\I)$ in \cite{withgunnar}. We predict that this will be true for all unital graph $C^*$-algebras with a unique nontrivial ideal, and describe the range of this invariant here.

\begin{theorem} \label{one-ideal-finite-vertices-range-thm}
If $C^*(E)$ is the $C^*$-algebra of a graph with a finite number of vertices that contains a unique nontrivial ideal $\I$, then the range of this invariant is all six-term exact sequences
$$\xymatrix{
G_1 \ar[r]^-{\epsilon} & G_2 \ar[r]^-{\gamma} & G_3 \ar[d]^-0 \\
F_3 \ar[u]^-\delta & F_2 \ar[l]_-{\gamma'} & F_1 \ar[l]_-{\epsilon'} 
}
$$
as in Theorem \ref{one-ideal-range-thm}, 
satisfying the further conditions:
\begin{enumerate}[(1)]
\item $F_1$, $F_3$, $G_1$, and $G_3$ are finitely generated abelian groups, 
\item $\rank F_1 \leq \rank G_1$ and $\rank F_3 \leq \rank G_3$,
\item if $(G_1,G_1^+)$ or $(G_3,G_3^+)$ is a Riesz group, that group is $(\Z,\Z^+)$ 
\end{enumerate}
The order unit $g_2$ of $K_0(C^*(E))$ can be any element of $G_2$ satisfying
\begin{enumerate}[(1)]
\addtocounter{enumi}{3}
\item if $(G_3,G_3^+)=(\Z,\Z^+)$ then $\gamma(g_2)>0$.
\end{enumerate}
Moreover, if $G_1 \cong \Z_{m_1} \oplus\ldots\oplus \Z_{m_k} \oplus \Z^m$ and $G_3 \cong \Z_{n_1}\oplus \ldots\oplus \Z_{n_l} \oplus \Z^n$, then $E$ may be chosen with $ m+k+n+l+2$ vertices. 
\end{theorem}
\begin{proof}
The necessary conditions from Theorem \ref{one-ideal-range-thm} are of course also relevant here, and we saw in the discussion preceding Proposition \ref{realize-unital-algebras} that (1) and (2) must hold in this case.
 In addition, we note that the only unital simple graph $C^*$-algebras that are AF are of the form $M_n(\C)$, proving that when $G_3$ is a Riesz group, it must be $\Z$. The same reasoning holds for $G_1$ since the ideal must be Morita equivalent to $M_n(\C)$. And finally, when $g_2=[1_{C^*(E)}]_0$ is given, we get that $\gamma(g_2)$ is given by the unit of the quotient, which is a strictly positive element of $K_0(M_n(\C))$ in case that is AF, proving necessity of (4).

To realize the invariant we argue separately for each case.

\noindent \textbf{Case $\inin$:}  For $i=1,3$ we find graphs $E_i$ realizing $G_i$ and  $F_i$ as above, but with the added assumptions that $|E_0^i|<\infty$ and that $[\mathbf 1]=\gamma(g_2)$ by appealing to Proposition \ref{realize-unital-algebras} rather than Proposition \ref{realize-simple-K-theories-prop}. We get that $C^*(E_i)$ is purely infinite and simple, as desired. The graph $E_3$ is right adhesive by condition (r1), so we may obtain $E_2$ by gluing $E_1$ and $E_3$ in a way that the obtained isomorphism $\alpha_2$ sends $\mathbf 1$ to $g_2$ because of Proposition \ref{unit-Y-prop}.

\noindent \textbf{Case $\fiin$:}  In this case $G_1=\Z$ and $F_1=0$, which we may realize by a graph $E_1$ with one vertex and no edges. We further realize $G_3$ and $F_3$ by $E_3$ as above, arranging so that $[\mathbf 1]=\gamma(g_2)$. Since $E_3$ is right adhesive, we may complete the argument as in the $\inin$ case.

\noindent \textbf{Case $\infi$:}  In this case $G_3=\Z$ and $F_3=0$, so that there is a splitting map $\sigma$ for $\gamma$. With $n>0$ the position of $\gamma(g_2)$ in $G_3=\Z$, we realize $G_3$ by a graph $E_1$ with two vertices $\{v,w\}$ and $n-1$ edges from $v$ to $w$ when $n>1$, or by a solitary vertex if $n=1$. We further realize $G_1$ and $F_1$ by $E_1$ as above, arranging so that $[\mathbf 1]=g_2-\sigma\circ\gamma(g_2)$. Letting $E_2$ be the union of $E_1$ and $E_3$ with infinitely many edges from the sink in $E_3$ to each vertex in $E_1$, we obtain a graph $C^*$-algebra $C^*(E_2)$  with a unique ideal.
Now appeal to Proposition \ref{split-Y-prop} to see that the invariant has the desired form.

\noindent \textbf{Case $\fifi$:}  
One checks by elementary methods that the graphs
\[
\xymatrix{
{\bullet}\ar[r]^x&{\bullet}\ar@{=>}[r]^\infty&{\bullet}\ar[r]^y&\bullet}
\]
generate all possible choices of order and unit in the given extension.
\end{proof}

Using basic group theory, and the fact that if $\phi : G \to H$ is a group homomorphism, then $\rank G = \rank \im \phi + \rank \ker \phi$, we see that the following relations are also satisfied:
\begin{itemize}
\item $F_2$ and $G_2$ are finitely generated abelian groups
\item $\rank F_1 \leq \rank F_2 \leq \rank F_1 + \rank F_3$
\item $\rank F_2  - \rank G_2= \rank F_3 - \rank G_3 + \rank F_1 - \rank G_1$ (so that, in particular, $\rank F_2 \leq \rank G_2$).
\end{itemize}

Recall that the class of Cuntz-Krieger algebras of matrices satisfying Condition~(II) coincides with the class of $C^*$-algebras of finite graphs with no sinks or sources that satisfy Condition~(K) (or, equivalently,  $C^*$-algebras of finite graphs with no sinks that have a finite number of ideals).  The following result can therefore be interpreted as determining the range of the six-term exact sequence for Cuntz-Krieger algebras with a unique nontrivial ideal.

\begin{theorem} \label{one-ideal-CK-range-thm}
If $C^*(E)$ is the $C^*$-algebra of a finite graph with no sinks or sources and with a unique nontrivial ideal $\I$, then $\Ksix (C^*(E),\I)$ and $[1_{C^*(E)}]$ is a complete isomorphism invariant within this class.  The range of this invariant is all six-term exact sequences
$$\xymatrix{
G_1 \ar[r]^-{\epsilon} & G_2 \ar[r]^-{\gamma} & G_3 \ar[d]^-0 \\
F_3 \ar[u]^-\delta & F_2 \ar[l]_-{\gamma'} & F_1 \ar[l]_-{\epsilon'} 
}
$$
satisfying the conditions (1),(3),(4) above as well as:
\begin{itemize}
\item[(2')] $\rank G_1 = \rank F_1$ and $\rank G_3 = \rank F_3$,
\end{itemize}
The order unit of $C^*(E)$ can be any element of $G_2$. Moreover, if $G_1 \cong \Z_{m_1} \oplus \ldots\oplus  \Z_{m_k} \oplus \Z^m$ and $G_3 \cong \Z_{n_1} \oplus\ldots\oplus \Z_{n_l} \oplus \Z^n$, then $E$ may be chosen with no more than $m+k+n+l+2$ vertices. 
\end{theorem}
\begin{proof}
Proceed as in the proof of  Case $\inin$ of Theorem~\ref{one-ideal-range-thm}, but use Proposition~\ref{realize-CK-algebras} in place of Proposition~\ref{realize-simple-K-theories-prop}. 
\end{proof}

\noindent From basic group theory, and using the fact that if $\phi : G \to H$ is a group homomorphism, then $\rank G = \rank \im \phi + \rank \ker \phi$, we see that the following relations are also satisfied:
\begin{itemize}
\item $F_2$ and $G_2$ are finitely generated abelian groups
\item $\rank F_1 \leq \rank F_2 \leq \rank F_1 + \rank F_3$
\item $\rank F_2 = \rank G_2$.
\end{itemize}

\section{Permanence}\label{permanence}

Consider a class of $C^*$-algebras $\mathfrak C$ with the property that
whenever $\A\in \mathfrak C$, any ideal $\I$ and any quotient $\A/\I$ also
lies in $\mathfrak C$. A \textbf{permanence result} for $\mathfrak C$ 
is a result that gives conditions for any extension 
\[
\xymatrix{{0}\ar[r]&\I \ar[r]^-{\iota}&\A \ar[r]^-{\pi}&\A/\I \ar[r]&0}
\]
to have the property that $\I,\A/\I\in \mathfrak C$ implies $\A\in \mathfrak C$, in terms of 
the six-term exact sequence 
\[
\xymatrix{K_0(\I) \ar[r]^-{\iota_*}&K_0(\A) \ar[r]^-{\pi_*}&K_0(\A/\I)
  \ar[d]^{\partial_0}\\
K_1(\A/\I)\ar[u]^-{\partial_1} &K_1(\A) \ar[l]^-{\pi_*}&K_1(\I) \ar[l]^-{\iota_*}}
\]
from $K$-theory.

Two well-known permanence results are of direct relevance for the
following. First, Brown in 
\cite{lgb:eafaplp} proved that if $\I$ and $\A/\I$ are AF algebras, then
so is $\A$. And second, it follows from 
\cite[Theorem~3.14 and Corollary~3.16]{lgbgkp:crrz} that if $\I$ and $\A/\I$ are of real rank
zero, then $\A$ is of real rank zero precisely when $\partial_0=0$.

We now set out to prove a permanence result for the class $\mathfrak C$
of stable graph $C^*$-algebras of real rank zero. It is known that when $\A\in
\mathfrak C$, then so is any ideal $\I$ and any quotient $\A/\I$. We offer
the following permanence result under the added assumptions that $\I$
and $\A/\I$ are either AF or simple, and $\I$ is a {stenotic}
ideal of $\A$.

Our strategy for doing so is as follows: Given an extension of graph $C^*$-algebras satisfying the necessary conditions, we build, using the results in the previous section, a graph realizing its $K$-theoretic data. And then we appeal to work of the first named author, Restorff and Ruiz to be able to prove by classification that the given extension $C^*$-algebra is in fact isomorphic to the one given by the constructed graph.

\begin{theorem}
Let 
\[
\xymatrix{0\ar[r]&C^*(E_1)\ar[r]&\A\ar[r]&C^*(E_3)\ar[r]&0}
\]
be a stenotic extension with $C^*(E_1)$ and $C^*(E_2)$ both stable and either simple or AF. The following are equivalent
\begin{enumerate}[(i)]
\item $\A$ is a graph $C^*$-algebra 
\item $\A$ is a graph $C^*$-algebra of real rank zero
\item 
\begin{enumerate}[(1)]
\item $\partial_0=0$; and
\item $
K_0(C^*(E_3))^+=K_0(C^*(E_3))\Longrightarrow K_0(\A)^+=K_0(\A)$
\end{enumerate}
\end{enumerate}
\end{theorem}
\begin{proof} 
We first note that by Brown's extension result combined with Proposition \ref{realize-simple-K-theories-AF-prop}, all of the claims hold true in the $\fifi$ case, (iii)(2) being vacuously true.

Turning to the remaining three cases, let $\I=C^*(E_1)$ considered an ideal of $\A$. In these cases, either $\I$ or $\A/\I$ is simple, so that $\I$ is necessarily gauge-invariant. As we have seen, this forces $\partial_0=0$, and since both $\I$ and $\A/\I$ have real rank zero, we conclude the same about $\A$, proving $(i)\Longrightarrow (ii)$. That $(ii)\Longrightarrow (iii)(1)$ is also clear from \cite{lgbgkp:crrz}, and that $(ii)\Longrightarrow (iii)(2)$ follows from Corollary \ref{graph-ext-order}(2) since when $K_0(C^*(F))$ is trivially ordered, $C^*(F)$ must in our case be simple and purely infinite.

In the remaining and most challenging direction, we first
note that $\A$ is stable by appealing to the corona
factorization property. In case $\inin$, we know that $K_0(\A)$ is trivially ordered, and may hence use Theorem \ref{one-ideal-range-thm} to realize $\Ksix(\A,\I)$ as an ordered group by some graph $E_3$. Since $M(\I)/\I$ is simple the extension is automatically full, and by \cite{ERR} we get that $\A\simeq C^*(E_3)$. The  case $\fiin$ is solved precisely the same way by appealing instead to Theorem \ref{smallest-ideal-range-thm}. 

In case $\fiin$ the result follows from Theorem~\ref{largest-ideal-range-thm}. This time the extension is not full a priori, but turns out to be so because of condition (iii)(2) combined with \cite[Corollary 3.17]{ERR2}, so again we obtain the desired result. 
\end{proof}

The condition (iii)(2) is vacuously true in the $\fifi$ and $\fiin$ cases, and automatically true in the $\inin$ case as seen in Proposition \ref{inx-case}. It is necessary in the $\infi$ case as noted in \cite[Example 4.3]{ERR2}.

 The class of unital graph $C^*$-algebras is not closed under taking
 ideals. Nevertheless, we have the following.

\begin{theorem}
Let 
\[
\xymatrix{0\ar[r]&C^*(E_1)\otimes{\mathbb K}\ar[r]&\A\ar[r]&C^*(E_3)\ar[r]&0}
\]
be a unital essential extension with $C^*(E_1)$ and $C^*(E_3)$ both unital, simple and purely
infinite $C^*$-algebras. The following are equivalent
\begin{enumerate}[(i)]
\item $\A$ is a graph $C^*$-algebra 
\item $\A$ has real rank zero
\end{enumerate}
\end{theorem}
\begin{proof}
That (i) implies (ii) follows as above. To prove the other implication, set $\I=C^*(E_1)\otimes\K$ and first realize $\Ksix(\A,\I)$ along with the given element $[1_\A]_0$ by some finite graph $E_2$, using Theorem \ref{one-ideal-finite-vertices-range-thm}. By \cite{withgunnar}, we get that $\A\simeq C^*(E_2)$.
\end{proof}

\noindent Similarly, by Theorem \ref{one-ideal-CK-range-thm} we obtain the following result.

\begin{theorem}
Let 
\[
\xymatrix{0\ar[r]&\mathcal O_{B_1}\otimes{\mathbb
    K}\ar[r]&\A\ar[r]&\mathcal O_{B_3}\ar[r]&0}
\]
be a unital essential extension with $\mathcal O_{B_1}$ and $\mathcal O_{B_3}$ both simple
Cuntz-Krieger algebras. The following are equivalent
\begin{enumerate}[(i)]
\item $\A$ is a Cuntz-Krieger algebra 
\item $\A$ is a graph $C^*$-algebra
\item $\A$ has real rank zero
\end{enumerate}
\end{theorem}

We see no reason why this theorem should not hold when $\mathcal O_{B_1}$ and $\mathcal O_{B_3}$ are given of real rank zero and with an arbitrary ideal lattice, but at the moment proving this seems outside reach. Substantial progress has been reported in \cite{abk}.


\begin{thebibliography}{99}


\bibitem{amp} P. Ara, M.A. Moreno, and E. Pardo, \emph{Nonstable $K$-theory for graph algebras},
 Algebr. Represent. Theory \textbf{10} (2007), 157--178.

\bibitem{abk} S. Arklint, R. Bentmann, and T. Katsura, \emph{Reduction of filtered $K$-theory and a characterization of
Cuntz-Krieger algebras}, preprint.

\bibitem{BHRS}
T.~Bates, J.~H.~Hong, I.~Raeburn, and W.~Szyma\'nski,
\emph{The ideal structure of the $C^*$-algebras of infinite graphs},
Illinois J.~Math \textbf{46}  (2002), 1159--1176.

\bibitem{lgbgkp:crrz}
L.G. Brown and G.K. Pedersen, \emph{{$C^*$}-algebras of real rank zero}, J.\
  Funct.\ Anal. \textbf{99} (1991), 131--149.

\bibitem{lgb:eafaplp}
L.G. Brown, \emph{Extensions of {$AF$}-algebras: the projection lifting
  problem}, Operator Algebras and Applications: Symp. Pure Math. \textbf{38}
  (1982), 175--176.


\bibitem{CET}
T.~M.~Carlsen, S. Eilers, and M.~Tomforde, \emph{Index maps in the $K$-theory of graph algebras}, to appear in J.\ K-Theory, DOI: 10.1017/is011004017jkt156.

\bibitem{DHS}
K.~Diecke, J.~H.~Hong, W.~Szyma\'nski,
\emph{Stable rank of graph algebras. Type I graph algebras and
their limits.},  Indiana Univ.~Math.~J. \textbf{52} (2003),
963--979.

\bibitem{Dri}
D.~Drinen, \emph{Viewing AF-algebras as graph algebras},
Proc. Amer. Math. Soc. \textbf{128} (2000), 1991--2000. 

\bibitem{DT1}
D.~Drinen and M.~Tomforde, \emph{The $C^*$-algebras of
arbitrary graphs}, Rocky Mountain J.~Math. \textbf{35} (2005),
105--135.

\bibitem{DT2}
D.~Drinen and M.~Tomforde, \emph{Computing $K$-theory and
Ext for graph $C^*$-algebras}, Illinois J. Math.
\textbf{46} (2002), 81--91.

\bibitem{EHS}
E.~Effros, D.~Handelman, C.~L.~Shen, \emph{Dimension groups and their affine representations}, 
Amer. J. Math. \textbf{102} (1980), 385--407. 

\bibitem{withgunnar}
S.~Eilers and G.~Restorff, \emph{On R\o rdam's classification of certain $C^*$-algebras with one nontrivial ideal}, S. Eilers and G. Restorff.
Operator algebras: The Abel symposium 2004, 87-96. Abel Symposia \textbf{1}, 
Springer-Verlag,
2006.

\bibitem{ERR}
S.~Eilers, G.~Restorff, and E.~Ruiz, \emph{Classifying $C^*$-algebras with both finite and infinite subquotients}, submitted for publication.

\bibitem{ERR2}
S.~Eilers, G.~Restorff, and E.~Ruiz, \emph{The ordered $K$-theory of a full extension}, submitted for publication. 

\bibitem{ET1}
S.~Eilers and M.~Tomforde, \emph{On the classification of nonsimple graph $C^*$-algebras}, Math. Ann. \textbf{346} (2010), 393--418. 

\bibitem{Ell2}
G.~Elliott, \emph{On the classification of inductive limits of
sequences of semisimple finite-dimensional algebras}, J.~Algebra
\textbf{38} (1976), 29--44.

\bibitem{goodhand}
K. R. Goodearl and D. E. Handelman, \emph{Stenosis in dimension groups and AF $C^*$-algebras},
J. Reine Angew. Math. \textbf{332} (1982),  1--98.

\bibitem{handelman} 
D. Handelman, \emph{Extensions for AF $C^*$ algebras and dimension groups}, Trans. Amer. Math.
Soc. \textbf{271} (1982), 537--573.

\bibitem{tkasmt:ragaelaua}
T.~Katsura, A.~Sims, and M.~Tomforde, \emph{Realizations of {AF}-algebras as
  graph algebras, Exel-Laca algebras, and ultragraph algebras}, 
J. Funct. Anal. \textbf{257} (2009), 1589--1620.


\bibitem{Kir3}
E.~Kirchberg, The Classification of Purely Infinite $C^*$-algebras using Kasparov's Theory,  to appear in the Fields Institute Communications series.

\bibitem{KR}
E.~Kirchberg and M.~R\o rdam, \emph{Nonsimple purely
infinite $C^*$-algebras}, Amer. J. Math. \textbf{122}
(2000), 637--666.

\bibitem{Lin}
H.~Lin, \emph{The simplicity of the quotient algebra $M(A)/A$ for a simple $C^*$-algebra}, Math. Scand. \textbf{65} (1989), 119--128. 

\bibitem{sm:h}
S.~Mac Lane, \emph{Homology}, Springer-Verlag, Berlin, {G\"ottingen},
  Heidelberg, 1963.



\bibitem{MN} 
R.~Meyer and R.~Nest, \emph{$C^*$-algebras over topological spaces: filtrated $K$-theory}, to appear in Canad.\ J.\ Math.

\bibitem{Ng}
P.~W.~Ng, \emph{The corona factorization property}. Operator theory, operator algebras, and applications, 97--110, Contemp. Math., \textbf{414}, Amer. Math. Soc., Providence, RI, 2006.

\bibitem{Phi}
N.~C.~Phillips, \emph{A classification theorem for nuclear
purely infinite simple $C^*$-algebras},  Doc. Math.
\textbf{5} (2000), 49--114. 

\bibitem{RS}
I.~Raeburn and W.~Szyma\'{n}ski, \emph{Cuntz-Krieger algebras of
infinite graphs and matrices}, Trans. Amer. Math. Soc.
\textbf{356} (2004), 39--59.

\bibitem{Res} G.~Restorff, \emph{Classification of
{Cuntz-Krieger} algebras up to stable isomorphism}, J.\ Reine Angew.\
Math. (2006), no.~598, 185--210.

\bibitem{Ro6}
M.~R\o rdam, \emph{Classification of extensions of certain $C^*$-algebras by their six term exact sequences in $K$-theory}, Math.~Ann. \textbf{308} (1997), 93--117.

\bibitem{mr:cnsc}
M.~R{\o}rdam, \emph{Classification of nuclear, simple {$C^*$}-algebras},
  Classification of nuclear $C^*$-algebras. Entropy in operator algebras,
  Encyclopaedia Math. Sci., vol. 126, Springer, Berlin, 2002, pp.~1--145.



\bibitem{Szy2}
W.~Szyma\'nski, \emph{The range of $K$-invariants for 
$C^*$-algebras of infinite graphs}, Indiana Univ. Math. J.
\textbf{51} (2002), 239--249.

\bibitem{Tomforde}
M. Tomforde, \emph{The ordered $K_0$-group of a graph $C^*$-algebra}, C. R.
Math. Acad. Sci. Soc. R. Can. \textbf{25} (2003), 19--25.

\bibitem{zhangdicho} S.~Zhang, \emph{A property of purely infinite simple $C^*$-algebras},
Proc. Amer. Math. Soc. \textbf{109} (1990), 717--720.

\end{thebibliography}
\end{document}